\documentclass[english, reqno]{amsart}
\usepackage{babel}
\usepackage{amsmath,amssymb}
\usepackage{caption}
\usepackage{tikz}
\usepackage[latin1]{inputenc}
\usepackage{color}
\usepackage{hyperref}

\def\vint_#1{\mathchoice%
          {\mathop{\kern 0.2em\vrule width 0.6em height 0.69678ex depth -0.58065ex
                  \kern -0.8em \intop}\nolimits_{\kern -0.4em#1}}%
          {\mathop{\kern 0.1em\vrule width 0.5em height 0.69678ex depth -0.60387ex
                  \kern -0.6em \intop}\nolimits_{#1}}%
          {\mathop{\kern 0.1em\vrule width 0.5em height 0.69678ex
              depth -0.60387ex
                  \kern -0.6em \intop}\nolimits_{#1}}%
          {\mathop{\kern 0.1em\vrule width 0.5em height 0.69678ex depth -0.60387ex
                  \kern -0.6em \intop}\nolimits_{#1}}}
                  
                  \newcommand{\aveint}[2]{\mathchoice%
          {\mathop{\kern 0.2em\vrule width 0.6em height 0.69678ex depth -0.58065ex
                  \kern -0.8em \intop}\nolimits_{\kern -0.45em#1}^{#2}}%
          {\mathop{\kern 0.1em\vrule width 0.5em height 0.69678ex depth -0.60387ex
                  \kern -0.6em \intop}\nolimits_{#1}^{#2}}%
          {\mathop{\kern 0.1em\vrule width 0.5em height 0.69678ex depth -0.60387ex
                  \kern -0.6em \intop}\nolimits_{#1}^{#2}}%
          {\mathop{\kern 0.1em\vrule width 0.5em height 0.69678ex depth -0.60387ex
                  \kern -0.6em \intop}\nolimits_{#1}^{#2}}}

\usepackage{epsfig}

\newcommand{\R}{\mathbb{R}}
\newcommand{\N}{\mathbb{N}}
\newcommand{\Z}{\mathbb{Z}}

\def\1{\raisebox{2pt}{\rm{$\chi$}}}

\newcommand{\abs}[1]{\left| #1 \right|}

\newcommand{\Om}{\Omega}

\newcommand{\dist}{\operatorname{dist}}

\newcommand{\M}{\mathcal{M}}

\newcommand{\tr}{{\rm tr}}

\renewcommand{\L}{\mathcal{L}}

\newcommand{\eps}{\varepsilon}

\theoremstyle{plain}
\newtheorem{definition}{Definition}[section]

\newtheorem{theorem}[definition]{Theorem}
\newtheorem*{theorem*}{Theorem}
\newtheorem{corollary}[definition]{Corollary}
\newtheorem{lemma}[definition]{Lemma}

\newtheorem{remark}[definition]{Remark}
\newtheorem{example}[definition]{Example}

\theoremstyle{definition}

\theoremstyle{remark}

\numberwithin{equation}{section}

\DeclareMathOperator*{\diam}{diam} %Diameter of a set
\renewcommand{\L}{\mathcal{L}}
\newcommand{\set}[2]{\{#1 \; :\; #2\}} % Set (two arguments)
\newcommand{\prodin}[2]{\langle #1,#2 \rangle} % Scalar product (two arguments)

\begin{document}

\title[Local regularity estimates]{Local regularity estimates for general discrete dynamic programming equations}

\author[Arroyo]{\'Angel Arroyo}
\address{MOMAT Research Group, Interdisciplinary Mathematics Institute, Department of Applied Mathematics and Mathematical Analysis, Universidad Complutense de Madrid, 28040 Madrid, Spain}
\email{ar.arroyo@ucm.es}

\author[Blanc]{Pablo Blanc}
\address{Department of Mathematics and Statistics, University of Jyv\"askyl\"a, PO~Box~35, FI-40014 Jyv\"askyl\"a, Finland}
\email{pablo.p.blanc@jyu.fi}

\author[Parviainen]{Mikko Parviainen}
\address{Department of Mathematics and Statistics, University of Jyv\"askyl\"a, PO~Box~35, FI-40014 Jyv\"askyl\"a, Finland}
\email{mikko.j.parviainen@jyu.fi}

\date{\today}

\keywords{ABP-estimate, elliptic non-divergence form partial differential  equation with bounded and measurable coefficients, dynamic programming principle, Harnack's inequality, local H\"older estimate, p-Laplacian, Pucci extremal operator, tug-of-war with noise} 
\subjclass[2010]{35B65, 35J15, 35J92,  91A50}

\maketitle

\begin{abstract}
We obtain an analytic proof for asymptotic H\"older estimate and Harnack's inequality for solutions to a discrete dynamic programming equation.  The results also generalize to functions satisfying Pucci-type inequalities for discrete extremal operators. Thus the results cover a quite general class of equations.
\end{abstract}

%%%%%%%%%%%%%%%%%%%%%%%%%%%

\section{Introduction}

Recently a quite general method for regularity of stochastic processes was devised in \cite{arroyobp}. It is shown that expectation  of a discrete stochastic process or equivalently a function satisfying the dynamic programming principle  (DPP)  
\begin{align}
\label{eq:dpp-intro}
u (x) =\alpha  \int_{\R^N} u(x+\eps z) \,d\nu_x(z)+\frac{\beta}{\abs{B_\eps}}\int_{B_\eps(x)} u(y)\,dy
+\eps^2 f(x),
\end{align}	
where $f$ is a Borel measurable bounded function and $\nu_x$ is a symmetric probability measure with rather mild conditions, is asymptotically H\"older regular. Moreover, the result generalizes to Pucci-type extremal operators and conditions of the form 
\begin{align}
\label{eq:pucci-extremals}
\mathcal L_\eps^+ u\ge -\abs{f},\quad \mathcal L_\eps^- u\le  \abs{f}, 
\end{align}
where $\mathcal L_\eps^+, \mathcal L_\eps^-$ are Pucci-type extremal operators related to operators of the form (\ref{eq:dpp-intro}) as in Definition~\ref{def:pucci}. As a consequence, the results immediately cover for example tug-of-war type stochastic games, which have been an object of a recent interest. 

The proof in \cite{arroyobp} uses probabilistic interpretation. In the PDE setting the closest counterpart would be Krylov-Safonov regularity method \cite{krylovs79}. It gives H\"older regularity of solutions and Harnack's inequality for      
elliptic equations with merely bounded and measurable coefficients.  The next natural question, and the aim of this paper, is to try to obtain an analytic proof. In the PDE setting the closest counterpart would be Trudinger's analytic proof of the Krylov-Safonov regularity result in \cite{trudinger80}. 

The H\"older estimate is obtained in Theorem \ref{Holder} (stated here in normalized balls for convenience) and it applies to (\ref{eq:dpp-intro}) by selecting $\rho=\sup\abs{f}$: 
\begin{theorem*}
There exists $\eps_0>0$ such that if $u$ satisfies $\L_\eps^+ u\ge -\rho$ and $\L_\eps^- u\le  \rho$  in $B_{2}$   where $\eps<\eps_0$, we have for suitable constants
\[
|u(x)-u(z)|\leq C\left(\sup_{B_{2}}|u|+\rho\right)\Big(|x-z|^\gamma+\eps^\gamma\Big)
\]
for every $x, z\in B_1$.
\end{theorem*}

 After establishing a H\"older regularity estimate, it is natural to ask in the spirit of Krylov, Safonov and Trudinger for Harnack's inequality. To the best of our knowledge, this was not known before in our context.  The regularity techniques in PDEs or in the nonlocal setting  utilize, heuristically speaking, the fact that there is information available in all scales. Concretely, a rescaling argument is used in those contexts in arbitrary small cubes.  In our case, discreteness sets limitations, and these limitations have some crucial effects. Indeed, the standard formulation of Harnack's inequality does not hold in our setting as we show by a counter example. 
Instead, we establish an asymptotic Harnack's inequality in Theorem \ref{Harnack}:
%%%%%%%%%%%%
\begin{theorem*}
There exists $\eps_0>0$ such that if $u$ satisfies $\L_\eps^+ u\ge -\rho$ and $\L_\eps^- u\le  \rho$ in $B_{7}$ where $\eps<\eps_0$, we have for suitable constants
\begin{equation*}
	\sup_{B_1}u
	\leq
	 C\left(\inf_{B_1}u+\rho+\eps^{2\lambda}\sup_{B_3}u\right).
\end{equation*}
\end{theorem*}
%%%%%%%%%%%%

 Both the asymptotic H\"older estimate and Harnack's inequality are stable when passing to a limit with the scale $\eps$, and we recover the standard H\"older estimate and Harnack's inequality in the limit.

The key point in the proof is to establish the De Giorgi type oscillation estimate that roughly states the following (here written for the zero right hand side and suitable scaling for simplicity):  
Under certain assumptions if $u$ is a (sub)solution to (\ref{eq:dpp-intro}) with $u\leq 1$ in a suitable bigger ball and
\[
|B_{R}\cap \{u\leq 0\}|\geq \theta |B_R|,
\]
for some $\theta>0$, then there exist $\eta>0$ such that
\[
\sup_{B_R} u \leq 1-\eta.
\]
This is established in Lemma~\ref{DeGiorgi}. Then we can obtain asymptotic H\"older continuity by a finite iteration combined with a rough estimate in the scales below $\eps$.

It is not straightforward to interpret the probabilistic proof in \cite{arroyobp} into analytic form to obtain the proof of Lemma~\ref{DeGiorgi}. Instead, we need to devise an iteration for the level sets
\[
A=\{u\geq K^k\} \quad \text{ and }\quad B=\{u\geq K^{k-1}\}.
\]
It seems difficult to produce an estimate between the measures of $A$ and $B$ by using the standard version of the Calder\'on-Zygmund  decomposition. The equation (\ref{eq:dpp-intro}) is not infinitesimal, but if we simply drop all the cubes smaller than of scale $\eps$ in the  decompositions, we have no control on the size of the error. To treat this, we use an additional condition for selecting additional cubes of scale $\eps$. On the other hand, additional cubes should belong to the set $B$ above, so there are two competing objectives. Different nonlocal analytic arguments, Alexandrov-Bakelman-Pucci (ABP) type estimates, and suitable cut-off levels will be used.

Unfortunately, but necessarily, the additional condition produces an error term in the estimate between measures of $A$ and $B$. Nonetheless, we can accomplish the level set measure estimate in Lemma \ref{measure bound} which is sufficient to get the De Giorgi oscillation lemma.

The H\"older estimate and Harnack's inequality are key results in the theory of non-divergence form elliptic partial differential equations with  bounded and measurable coefficients. They were first obtained by Krylov and Safonov in \cite{krylovs79, krylovs80} by stochastic arguments. Later, an analytic proof for strong solutions was established by Trudinger in \cite{trudinger80}, see also \cite[Section 9]{gilbargt01}. In the case of viscosity solutions for fully nonlinear elliptic equations, the ABP estimate and Harnack's inequality were obtained by  Caffarelli \cite{caffarelli89}, also covered in \cite[{Chapters 3 and 4}]{caffarellic95}. For nonlocal equations, such results have been considered more recently for example in \cite{caffarellis09} or \cite{caffarellitu20}. In the case of fully discrete difference equations, we refer the reader to \cite{kuot90}. 

There is a classical well-known connection between the Brownian motion and the Laplace equation.  The dynamic programming principle (\ref{eq:dpp-intro}) is partly motivated by the connection of stochastic processes with the $p$-Laplace equation and other nonlinear PDEs. Our results cover (see \cite{arroyobp} for details) in particular a stochastic two player game called the tug-of-war game with noise. The tug-of-war game and its connection with the infinity Laplacian was discovered in \cite{peresssw09}. For the tug-of-war games with noise and their connection to $p$-Laplacian, see for example \cite{peress08}, \cite{manfredipr12}, \cite{blancr19} and \cite{lewicka20}.
There are several regularity methods devised for tug-of-war games with noise: in the early papers a global approach based on translation invariance was used. Interior a priori estimates were obtained in \cite{luirops13} and \cite{luirop18}. However, none of these methods seem to directly apply in the general setup of this paper. In this setup, we refer to probabilistic approaches in \cite{arroyobp} and with additional distortion bounds in \cite{arroyop20}.

\tableofcontents

\section{Preliminaries}
\label{preliminaries}

Let  $\Lambda\geq 1$, $\eps>0$, $\beta\in (0,1]$ and $\alpha=1-\beta$.
Constants may depend on $\Lambda$, $\alpha$, $\beta$ and the dimension $N$. 
Further dependencies are specified later. 

Throughout the article $\Omega \subset \mathbb{R}^N$ denotes a bounded domain, and $B_r(x)=\{y\in\R^N:|x-y|<r\}$ as well as $B_r=B_r(0)$.
We use $\N$ to denote the set of positive integers.
We define an extended domain as follows
\begin{equation*}
		\widetilde\Omega_{\Lambda\varepsilon}
		:\,=
		\set{x\in\R^n}{\dist(x,\Omega)<\Lambda\varepsilon}.
\end{equation*}
We further denote
\[
\int u(x)\,dx=\int_{\R^N} u(x)\,dx
\quad
\text{ and }
\quad
\vint_A u(x)\,dx=\frac{1}{|A|}\int_{A} u(x)\,dx.
\]
Moreover,
\[
\|f\|_{L^N(\Omega)}=\left(\int_\Omega |f(x)|^N\,dx\right)^{1/N}
\]
and
\[
\|f\|_{L^\infty(\Omega)}=\sup_\Omega |f|.
\]
When no confusion arises we just simply denote $\|\cdot\|_N$ and $\|\cdot\|_\infty$, respectively.

For  $x=(x_1,\ldots,x_n)\in\R^N$ and $r>0$, we define $Q_r(x)$ the open cube of side-length $r$ and center $x$ with faces parallel to the coordinate hyperplanes. In other words,
\begin{equation*}
	Q_r(x)
	:\,=
	\{y\in\R^N\,:\,|y_i-x_i|<r/2,\ i=1,\ldots,n\}.
\end{equation*}
In addition, if $Q=Q_r(x)$ and $\ell>0$, we denote $\ell Q=Q_{\ell r}(x)$.

Let $\M(B_\Lambda)$ denote the set of symmetric unit Radon measures with support in $B_\Lambda$ and $\nu:\R^N\to \M(B_\Lambda)$
such that
\begin{equation}\label{measurable-nu}
x\longmapsto\int u(x+z) \,d\nu_x(z)
\end{equation}
defines a Borel measurable function for every Borel measurable  $u:\R^N\to \R$.
By symmetric, we mean
\begin{align*}
\nu_x(E)=\nu_x(-E).
\end{align*}
for every measurable set $E\subset\R^N$.

It is worth remarking that the hypothesis \eqref{measurable-nu} on Borel measurability holds, for example, when the $\nu_x$'s are the pushforward of a given probability measure $\mu$ in $\R^N$. More precisely, if there exists a Borel measurable function $h:\R^N\times \R^N\to B_\Lambda$ such that 
\[
\nu_x=h(x,\cdot)\#\mu
\]
for each $x$, then
\[
\begin{split}
v(x)
&=\int u(x+z) \,d\nu_x(z)\\
&=\int u(x+h(x,y)) \,d\mu(y)\\
\end{split}
\]
is measurable by Fubini's theorem.

We consider here solutions to the Dynamic Programming Principle (DPP) given by
\begin{align*}
u (x) =\alpha  \int u(x+\eps v) \,d\nu_x(v)+\beta\vint_{B_\eps(x)} u(y)\,dy+\eps^2 f(x).
\end{align*}

\begin{definition}
\label{def:solutions}
We say that a bounded Borel measurable function $u$ is a subsolution to the DPP if it satisfies
\[
u (x)\leq\alpha  \int u(x+\eps z) \,d\nu_x(z)+\beta\vint_{B_\eps(x)} u(y)\,dy+\eps^2 f(x)
\]
in $\Omega$.
Analogously, we say that $u$ is a supersolution if the reverse inequality holds.
If the equality holds, we say that it is a solution to the DPP.
\end{definition}

If we rearrange the terms in the DPP,  we may alternatively use a notation that is closer to the difference methods.
\begin{definition}
Given a Borel measurable bounded function $u:\R^N\to \R$, we define $\L_\eps u:\R^N\to \R$ as
\[
\L_\eps u(x)=\frac{1}{\eps^2}\left(\alpha  \int u(x+\eps z) \,d\nu_x(z)+\beta\vint_{B_\eps(x)} u(y)\,dy-u(x)\right).
\]
With this notation,  $u$ is a subsolution (supersolution) if and only if $\L_\eps u+f \geq 0 (\leq 0)$.
\end{definition}

By defining 
\begin{align}
\label{eq:delta}
\delta u(x,y):\,=u(x+y)+u(x-y)-2u(x),
\end{align}
and recalling the symmetry condition on $\nu_x$ we can rewrite
\begin{equation*}
\L_\eps u(x)=\frac{1}{2\eps^2}\left(\alpha  \int \delta u(x,\eps z) \,d\nu_x(z)+\beta\vint_{B_1} \delta u(x,\eps y)\,dy\right).
\end{equation*}

Our theorems actually hold for functions merely satisfying Pucci-type inequali-
ties. 

\begin{definition}
\label{def:pucci}
Let $u:\R^N\to\R$ be a bounded Borel measurable function. We define the extremal Pucci type operators
\begin{equation}\label{L-eps+}
\begin{split}
	\L_\eps^+ u(x)
	:\,=
	~&
	\frac{1}{2\eps^2}\bigg(\alpha \sup_{\nu\in \M(B_\Lambda)} \int \delta u(x,\eps z) \,d\nu(z)	+\beta\vint_{B_1} \delta u(x,\eps y)\,dy\bigg)
	\\
	=
	~&
	\frac{1}{2\eps^2}\bigg(\alpha \sup_{z\in B_\Lambda} \delta u(x,\eps z) +\beta\vint_{B_1} \delta u(x,\eps y)\,dy\bigg)
\end{split}
\end{equation}
and
\begin{equation}\label{L-eps-}
\begin{split}
	\L_\eps^- u(x)
	:\,=
	~&
	\frac{1}{2\eps^2}\bigg(\alpha \inf_{\nu\in \M(B_\Lambda)} \int \delta u(x,\eps z) \,d\nu(z)	+\beta\vint_{B_1} \delta u(x,\eps y)\,dy\bigg)
	\\
	=
	~&
	\frac{1}{2\eps^2}\bigg(\alpha \inf_{z\in B_\Lambda} \delta u(x,\eps z) +\beta\vint_{B_1} \delta u(x,\eps y)\,dy\bigg),
\end{split}
\end{equation}
where $\delta u(x,\eps y)=u(x+\eps y)+u(x-\eps y)-2u(x)$ for every $y\in B_\Lambda$.
\end{definition}

More generally we can consider functions that satisfy 
\begin{equation*}
\L_\eps^+ u\ge -\rho,\quad \L_\eps^- u\le  \rho.
\end{equation*}
If we omit the notation above $\L_\eps^- u\le  \rho$ reads as
\begin{align*}
u (x) 
	&
	\geq
	\alpha\inf_{\nu\in \M(B_\Lambda)}  \int u(x+\eps v) \,d\nu (v)+\beta\vint_{B_\eps(x)} u(y)\,dy-\eps^2\rho.
\end{align*}

\sloppy
Observe that the natural counterpart for the Pucci operator
$P^+(D^2u)=\sup_{I \leq A\leq \Lambda I}\tr(AD^2u)$
is given by
\begin{align}
\label{eq:about-extremal-operators}
	P_\eps^+ u(x)
	:\,=
	\frac{1}{2\eps^2}\sup_{I\leq A\leq \Lambda I}\vint_{B_1} \delta u(x,\eps A y)\,dy.
\end{align}
Our operator is extremal in the sense that we have $\L_{\eps}^+ u\geq P_\eps^+ u$ for $\beta=\frac{1}{\Lambda^N}.$

In many places we consider $u$ defined in the whole $\R^N$ but only for expository reasons: we need always have the function defined in a larger set than where the equation is given so that the integrands in the operators are defined; this we always assume. 

The existence of solutions to the  DPP can be seen by Perron's method. 
For the uniqueness in \cite{arroyobp} we employed the connection to a stochastic process.
Here we give a pure analytic proof of the uniqueness.

\begin{lemma}[Existence and uniqueness]
There exists a unique solution to the DPP with given boundary values.
\end{lemma}

\begin{proof}
As stated, the existence can be proved by Perron's method. 
Then, there is a maximal solution that we denote $u$.
Suppose that there is another solution $v$.
We have $v\leq u$ and our goal is to show that equality holds.

We define
\[
M=\sup_{x\in\Omega}u(x)-v(x)
\]
and assume, for the sake of contradiction, that $M>0$.
We define
\[
A
=\frac{|\{y\in B_\eps(x): \pi_1(y)>\pi_1(x)+\eps/2\}|}{\abs{B_\eps}}
=\frac{|\{y\in B_1: \pi_1(y)>1/2\}|}{\abs{B_1}}
\]
where $\pi_1$ stands for the projection in the first coordinate.

Given $\delta>0$ we consider $x_0\in \Omega$ such that $u(x_0)-v(x_0)>M-\delta$.
We have
\[
\begin{split}
M-\delta
&<u(x_0)-v(x_0)\\
&=\alpha  \int u(x_0+\eps z)-v(x_0+\eps z) \,d\nu_{x_0}(z)+\beta\vint_{B_\eps(x_0)} u(y)-v(y)\,dy\\
&<\alpha  M +\beta (1-A) M+\beta A \vint_{\{y\in B_\eps(x_0): \pi_1(y)>\pi_1(x_0)+\eps/2\}} u(y)-v(y)\,dy.
\end{split}
\]
Simplifying we obtain
\[
M-\frac{\delta}{\beta A}
<\vint_{\{y\in B_\eps(x_0): \pi_1(y)>\pi_1(x_0)+\eps/2\}} u(y)-v(y)\,dy.
\]
Then, there exists $x_1\in \{y\in B_\eps(x_0): \pi_1(y)>\pi_1(x_0)+\eps/2\}$ such that
\[
M-\frac{\delta}{\beta A}<u(x_1)-v(x_1).
\]
Inductively, given $x_{k-1}\in \Omega$ we construct $x_k$ such that $M-\frac{\delta}{(\beta A)^k}<u(x_k)-v(x_k)$ and $\pi_1(x_k)>\pi_1(x_0)+k\eps/2$.
Since $\Omega$ is bounded and the first coordinate increases in at least $\eps/2$ in every step, there exists a first $n$ such that $x_n\not\in \Omega$.
Observe that $n\leq n_0=\frac{\diam(\Omega)}{\eps/2}$, therefore for $\delta$ small enough such that $M-\frac{\delta}{(\beta A)^{n_0}}>0$ we have reached a contradiction. 
In fact, we have
\[
0
< M-\frac{\delta}{(\beta A)^{n_0}}
\leq M-\frac{\delta}{(\beta A)^{n}}
\leq  u(x_n)-v(x_n)
\]
and $u(x_n)=v(x_n)$ since $x_n\not\in\Omega$.
\end{proof}

%%%%%%%%%%%%%%%%%%%%%%%%%%%%%%%%
%%%%%%%%%%%%%%%%%%%%%%%%%%%%%%%%

\subsection{Examples and connection to PDEs}

In this section, we recall some examples from \cite{arroyobp} alongside other ones, all of which are covered by our results.
First, we comment  about passage to the limit with the step size $\eps$ where the connection to PDEs arises.

We consider $\phi\in C^2(\Omega)$, and  use the second order Taylor's expansion of $\phi$ to obtain
\begin{equation*}
	\lim_{\eps\to 0}\L_\eps\phi(x)
	=
	\mathrm{Tr}\{D^2\phi(x)\, A(x)\},
\end{equation*}
where 
\begin{equation*}
	A(x)
	:\,=
	\frac{\alpha}{2}\int z\otimes z\,d\nu_x(z)+\frac{\beta}{2(N+2)}\, I.
\end{equation*}
Above $a\otimes b$ stands for the tensor product of vectors $a,b\in\R^n$, that is, the matrix with entries $(a_ib_j)_{ij}$.
See Example 2.3 in \cite{arroyobp} for the details.

We have obtained a linear second order partial differential operator.
Furthermore, for $\beta\in(0,1]$, the operator is uniformly elliptic: given $\xi\in\R^N\setminus\{0\}$,
we can estimate
\begin{equation*}
	\frac{\beta}{2(N+2)}
	\leq
	\frac{\langle A(x) \xi,\xi\rangle}{|\xi|^2}
	\leq
	\frac{\alpha\Lambda^2}{2}+\frac{\beta}{2(N+2)}.
\end{equation*}
Roughly speaking, in the DPP (\ref{eq:dpp-intro}), the fact that $\beta$ is strictly positive corresponds to the concept of uniform ellipticity in PDEs. In stochastic terms, there is always certain level of diffusion to each direction.

It also holds, using Theorem~\ref{Holder} (cf.\ \cite[Theorem 4.9]{manfredipr12}), that under suitable regularity assumptions, the solutions $u_{\eps}$  to the DPP converge to a viscosity solution $v\in C(\Om)$ of 
\begin{align*}
\mathrm{Tr}\{D^2 v(x)\,A(x)\}=f(x),
\end{align*}
as $\eps\to 0$. This is obtained through the asymptotic Arzel\`a-Ascoli theorem \cite[Lemma 4.2]{manfredipr12}.

Moreover, by passing to the limit under suitable uniqueness considerations we obtain that
the results in this paper imply the corresponding regularity for the solutions to the limiting PDEs.
That is we obtain that the limit functions are H\"older continuous and verify the classical Harnack inequality, see
Remark \ref{harnack:limit}.

The extremal inequalities (\ref{eq:pucci-extremals}) cover a wide class of discrete operators, comparable to the uniformly elliptic operators in PDEs covered by the Pucci extremal operators, see for example \cite{caffarellic95}.
Also recall (\ref{eq:about-extremal-operators}) where we commented on this connection.

%%%%%%%%%%
\begin{example}
Our result  applies to solutions of the nonlinear DPP given by
\[
u(x)=\alpha \sup_{\nu\in B_\Lambda} \frac{u(x+\eps \nu)+u(x-\eps \nu)}{2}
+\beta\vint_{B_1} u(x+\eps y)\,dy.
\]
In \cite{brustadlm20} a control problem associated to the nonlinear example is presented and, in the limit as $\eps \to 0$, a local PDE involving the dominative $p$-Laplacian operator arises.

Heuristically, the above DPP can be understood by considering a value $u$ at $x$, which can be computed by summing up different outcomes with corresponding probabilities: either a maximizing controller who gets to choose $\nu$ wins (probability $\alpha$), or a random step occurs (with probability $\beta$) within a ball of radius $\varepsilon$. If the controller wins, the position moves to $x+\eps \nu$ (with probability $1/2$) or to $x-\eps \nu$ (with probability $1/2$). 
\end{example}
%%%%%%%%%%

%%%%%%%%%%
\begin{example}
Motivation for this article partly arises from tug-of-war games.
In particular, the tug-of-war with noise associated to the DPP 
\begin{align}
\label{eq:p-dpp}
u(x)=\frac{\alpha}{2}\left( \sup_{B_\eps(x)} u + \inf_{B_ \eps(x) } u\right)+\beta \vint_{B_\eps(x)} u(z) dz + \eps^2 f(x).
\end{align}
was introduced in \cite{manfredipr12}.
This can be rewritten as 
\begin{align*}
\frac{1}{2\eps^2}\left(\alpha\left( \sup_{B_\eps(x)} u + \inf_{B_ \eps(x) } u	-2u(x)\right)+\beta\vint_{B_1} \delta u(x,\eps y)\,dy\right)+f(x)=0.
\end{align*}
Since 
\[
\sup_{B_\eps(x)} u + \inf_{B_\eps(x)} u \leq \sup_{z\in B_1}\big(  u(x+\eps z) +u(x-\eps z)\big)
\]
we have $0\leq f+ \L^+_\eps u$ and similarly   $0\ge f+ \L^-_\eps u$.
Therefore solutions to \eqref{eq:p-dpp} satisfy (\ref{eq:pucci-extremals}), and our results apply to these functions.
As a limit one obtains the $p$-Laplacian problem with  $2<p<\infty$.
See Example 2.4 in \cite{arroyobp} for other DPPs related to the $p$-Laplacian.
\end{example}
%%%%%%%%%%

\begin{example}
Consider a stochastic process where a particle jumps to a point in an ellipsoid $\eps E_x$ uniformly at random ($B_1\subset E_x\subset B_\Lambda$), see \cite{arroyop20}.
Such a process is associated to the DPP
\[
u(x)
=
\vint_{E_x} u(x+\eps y)\,dy.
\]
That DPP is covered by our results, see Example 2.7 in \cite{arroyobp}.
Such mean value property has been studied in connection with smooth solutions to PDEs in \cite{puccit76} by Pucci and Talenti.
\end{example}

%%%%%%%%%%
\begin{example}
 Also Isaacs type dynamic programming principle 
\[
u(x)=\alpha \sup_{V \in \mathcal V}\inf_{\nu \in V}  \frac{u(x+\eps \nu)+u(x-\eps \nu)}{2}
+\beta\vint_{B_1} u(x+\eps y)\,dy,
\]
with  $\mathcal V \subset\mathcal P(B_\Lambda)$, a subset of the power set, and $\beta>0$ can be mentioned as an example.
In particular, if we consider
\[
\mathcal V =\{\pi\cap B_\Lambda: \text{$\pi$ is an hyperplane of dimension $k$}\}
\]
 we obtain 
 $$ 
 \lambda_k (D^2u) +C \Delta u = f, 
 $$
 as a limiting PDE, where 
 $$
 \lambda_k (D^2 u) = \inf_{dim (V) = k} \sup_{v \in V} \langle D^2u \, v, v \rangle
 $$
 is the $k-$th eigenvalue of $D^2u$, see also \cite{blancr19b}.

\end{example}

 The applicability of the results in this article is by no means limited to these examples, but rather they apply to many kind of fully nonlinear uniformly elliptic PDEs.
%%%%%%%%%%

\section{Measure estimates}

One of the key ingredients in the proof of H\"older regularity is the measure estimate Lemma~\ref{first}. To prove it, we need an $\eps$-ABP estimate Theorem~\ref{eps-ABP}, an estimate for the difference between $u$ and its concave envelope Corollary~\ref{estimate Q}, as well as a suitable barrier functions Lemma~\ref{barrier}.

\subsection{The $\eps$-ABP estimate}

Next we recall a version of the ABP estimate. The discrete nature of our setting forces us to consider non-continuous subsolutions of the DPP, so the corresponding concave envelope $\Gamma$ might not be $C^{1,1}$ as in the classical setting. Moreover, in this setting it is not easy to use the change of variables formula for the integral to prove the ABP.
In our previous work \cite{arroyobp}, the ABP estimate (Theorem~\ref{eps-ABP} below) is adapted to the discrete {$\eps$-setting} following an argument by Caffarelli and Silvestre (\cite{caffarellis09}) for nonlocal equations.  The idea is to use a covering argument on the contact set (where $u$ coincides with $\Gamma$) to estimate the oscillation of $\Gamma$. It is also interesting to note that one can recover the classical ABP estimate by taking limits as $\eps\to 0$.

However, the $\eps$-ABP estimate as stated in \cite{arroyobp} turns out to be insufficient to establish the preliminary measure estimates needed in our proof of H\"older regularity. To deal with this inconvenience, and since the $\eps$-ABP exhibits certain independence of the behavior of $u$ outside the contact set, we need to complement the $\eps$-ABP estimate with an estimate (in measure) of the difference between the subsolution $u$ and its concave envelope $\Gamma$ (Lemma ~\ref{estimate B_eps}) in a neighborhood of any contact point.

Given $\eps>0$, we denote by $\mathcal{Q}_\eps(\R^N)$ a grid of open cubes of diameter $\eps/4$ covering $\R^N$ up to a measure zero. Take
\begin{equation*}
	\mathcal{Q}_\eps(\R^N)
	:\,=
	\set{Q=Q_{\frac{\eps}{4\sqrt{N}}}(x)}{x\in\frac{\eps}{4\sqrt{N}}\,\Z^N}.
\end{equation*}
In addition, if $A\subset\R^N$ we write
\begin{equation*}
	\mathcal{Q}_\eps(A)
	:\,=
	\set{Q\in\mathcal{Q}_\eps(\R^N)}{\overline Q\cap A\neq\emptyset}.
\end{equation*}

In order to obtain the measure estimates, given a bounded Borel measurable function $u$ satisfying the conditions in Theorem~\ref{eps-ABP}, we define the concave envelope of  $u^+=\max\{u,0\}$ in $B_{2\sqrt{N}+\Lambda\eps}$ as the function
\begin{equation*}
	\Gamma(x)
	:\,=
	\begin{cases}
	\inf\set{\ell(x)}{\text{for all hyperplanes } \ell\geq u^+ \text{ in } B_{2\sqrt{N}+\Lambda\eps}} & \text{ if } |x|<2\sqrt{N}+\Lambda\eps,
	\\
	0 & \text{ if } |x|\geq 2\sqrt{N}+\Lambda\eps.
	\end{cases}
\end{equation*}
Moreover, we define the superdifferential of $\Gamma$ at $x$ as the set of vectors
\begin{equation*}
	\nabla\Gamma(x)
	:\,=\set{\xi\in\R^N}{\Gamma(z)\leq\Gamma(x)+\prodin{\xi}{z-x}\ \text{ for all }\ |z|<2\sqrt{N}+\Lambda\eps}.
\end{equation*}
Since $\Gamma$ is concave, then $\nabla\Gamma(x)\neq\emptyset$ for every $|x|<2\sqrt{N}+\Lambda\eps$.

In addition, we define the contact set $K_u\subset\overline B_{2\sqrt{N}}$ as the set of points where $u$ and $\Gamma$ `agree':
\begin{equation*}
	K_u
	:\,=
	\set{|x|\leq 2\sqrt{N}}{\limsup_{y\to x}u(y)=\Gamma(x)}.
\end{equation*}
We remark that the set $K_u$ is compact. 
Indeed, $K_u$ is bounded and since $u\leq\Gamma$, the set of points where the equality is attained is given by $\limsup_{y\to x}u(y)-\Gamma(x)\geq 0$ and it is closed because $\limsup_{y\to x}u(y)-\Gamma(x)$ is upper semicontinuous.

Now we are in conditions of stating the $\eps$-ABP estimate, whose proof can be found in \cite[Theorem 4.1]{arroyobp} (see also Remark 7.4 in the same reference).

\begin{theorem}[$\eps$-ABP estimate]\label{eps-ABP}
Let $f\in C(\overline B_{2\sqrt{N}})$ and suppose that $u$ is a bounded Borel measurable function satisfying
\begin{equation*}
	\begin{cases}
	\L_\eps^+u+f\geq 0 & \text{ in } B_{2\sqrt{N}},
	\\
	u\leq 0 & \text{ in } \R^N\setminus B_{2\sqrt{N}},
	\end{cases}
\end{equation*}
where $\L_\eps^+u$ was defined in (\ref{L-eps+}).
Then
\begin{equation*}
	\sup_{B_{2\sqrt{N}}}u
	\leq
	C\bigg(\sum_{Q\in\mathcal{Q}_\eps(K_u)}(\sup_Qf^+)^N|Q|\bigg)^{1/N},
\end{equation*}
where $C>0$ is a constant independent of $\eps$.
\end{theorem}

All relevant information of $u$ in the proof of the $\eps$-ABP estimate turns out to be transferred to its concave envelope $\Gamma$ in the contact set $K_u$, while the behavior of $u$ outside $K_u$ does not play any role in the estimate. Therefore, in order to control the behavior of $u$ in $B_{2\sqrt{N}}$, in the next result we show that $u$ stays sufficiently close to its concave envelope in a large enough portion of the $\eps$-neighborhood of any contact point $x_0\in K_u$. It is also worth remarking that the result can be regarded as a refinement of Lemma 4.4 in \cite{arroyobp}, the main difference being the possible discontinuities that $u$ might present.

\begin{lemma}\label{estimate B_eps}
Under the assumptions of Theorem~\ref{eps-ABP}, let $x_0\in K_u$. Then for every $C>0$ large enough there exists $c>0$ such that
\begin{equation*}
	|B_{\eps/4}(x_0)\cap\{\Gamma-u\leq Cf(x_0)\eps^2\}|
	\geq
	c\eps^N.
\end{equation*}
\end{lemma}

\begin{proof}
By the definition of the set $K_u$, given $x_0\in K_u$ there exists a sequence $\{x_n\}_n$ of points in $\overline B_{2\sqrt{N}}$ converging to $x_0$ such that
\begin{equation*}
	\Gamma(x_0)
	=
	\lim_{n\to\infty}u(x_n).
\end{equation*}
Recall the notation $\delta u(x_n,y):\,=u(x_n+y)+u(x_n-y)-2u(x_n)$. Then, since $u\leq\Gamma$,
\begin{equation*}
\begin{split}
	\delta u(x_n,y)
	\leq
	~&
	\delta\Gamma(x_n,y)+2[\Gamma(x_n)-u(x_n)]
	\\
	\leq
	~&
	2[\Gamma(x_n)-u(x_n)],
\end{split}
\end{equation*}
for every $y$, where the concavity of $\Gamma$ has been used in the second inequality. In particular, 
\begin{equation*}
	\sup_{z\in B_\Lambda}\delta u(x_n,\eps z)
	\leq
	2[\Gamma(x_n)-u(x_n)]
	\longrightarrow
	0
\end{equation*}
as $n\to\infty$.
On the other hand,
\begin{equation*}
\begin{split}
	\frac{1}{2}\vint_{B_1}\delta u(x_n,\eps y)\,dy
	=
	~&
	\vint_{B_\eps}(u(x_n+y)-u(x_n))\,dy
	\\
	=
	~&
	\vint_{B_\eps}(u(x_0+y)-\Gamma(x_0))\,dy
	\\
	~&
	+\Gamma(x_0)-u(x_n)+\vint_{B_\eps}(u(x_n+y)-u(x_0+y))\,dy,
\end{split}
\end{equation*}
and taking limits
\begin{equation*}
	\lim_{n\to\infty}\frac{1}{2}\vint_{B_1}\delta u(x_n,\eps y)\,dy
	=
	\vint_{B_\eps}(u(x_0+y)-\Gamma(x_0))\,dy.
\end{equation*}
Replacing in the expression for $\L_\eps^+u(x_n)$ we get
\begin{equation*}
	\eps^2\liminf_{n\to\infty}\L_\eps^+u(x_n)
	\leq
	\beta\vint_{B_\eps}(u(x_0+y)-\Gamma(x_0))\,dy.
\end{equation*}
Since $\L_\eps^+u+f\geq 0$ by assumption with continuous $f$, we obtain
\begin{equation*}
\begin{split}
	\frac{f(x_0)\eps^2}{\beta}
	\geq
	~&
	\vint_{B_\eps}(\Gamma(x_0)-u(x_0+y))\,dy
	\\
	=
	~&
	\vint_{B_\eps}(\Gamma(x_0)-u(x_0+y)+\prodin{\xi}{y})\,dy,
\end{split}
\end{equation*}
for every vector $\xi\in\R^N$, where the equality holds because of the symmetry of $B_\eps$. Since $\nabla\Gamma(x_0)\neq\emptyset$ by the concavity of $\Gamma$, we can fix $\xi\in\nabla\Gamma(x_0)$.

Next we split $B_\eps$ in two sets: $B_\eps\cap\{\Phi\leq Cf(x_0)\eps^2\}$ and $B_\eps\cap\{\Phi>Cf(x_0)\eps^2\}$, where we have denoted
\begin{equation*}
	\Phi(y)
	:\,=
	\Gamma(x_0)-u(x_0+y)+\prodin{\xi}{y}
\end{equation*}
for every $y\in B_\eps$ for simplicity, and we study the integral of $\Phi$ over both subsets. 

First, since $u\leq\Gamma$ and $\xi\in\nabla\Gamma(x_0)$ we have that
\begin{equation*}
	\Phi(y)
	\geq
	\Gamma(x_0)-\Gamma(x_0+y)+\prodin{\xi}{y}
	\geq
	0
\end{equation*}
for every $y\in B_\eps$, so we can estimate
\begin{equation*}
\begin{split}
	\int_{B_\eps\cap\{\Phi\geq Cf(x_0)\eps^2\}}\Phi(y)\,dy
	\geq
	0.
\end{split}
\end{equation*}
On the other hand,
\begin{equation*}
	\int_{B_\eps\cap\{\Phi>Cf(x_0)\eps^2\}}\Phi(y)\,dy
	>
	|B_\eps\cap\{\Phi>Cf(x_0)\eps^2\}|Cf(x_0)\eps^2.
\end{equation*}
Summarizing, we have proven that
\begin{equation*}
	\frac{f(x_0)\eps^2}{\beta}
	>
	\frac{|B_\eps\cap\{\Phi>Cf(x_0)\eps^2\}|}{|B_\eps|}Cf(x_0)\eps^2,
\end{equation*}
so
\begin{equation*}
	|B_{\eps/4}\cap\{\Phi>Cf(x_0)\eps^2\}|
	\leq
	|B_\eps\cap\{\Phi>Cf(x_0)\eps^2\}|
	<
	\frac{|B_\eps|}{C\beta}
	=
	\frac{4^N}{C\beta}|B_{\eps/4}|.
\end{equation*}
Therefore,
\begin{equation*}
	|B_{\eps/4}\cap\{\Phi\leq Cf(x_0)\eps^2\}|
	\geq
	|B_{\eps/4}|\left(1-\frac{4^N}{C\beta}\right)
	=
	c\eps^N.
\end{equation*}
Finally, replacing $\Phi$, and since $\Gamma(x_0+y)\leq\Gamma(x_0)+\prodin{\xi}{y}$ for every $y\in B_{\eps/4}$ and $\xi\in\nabla\Gamma(x_0)$, we can estimate
\begin{equation*}
\begin{split}
	c\eps^N
	\leq
	~&
	|\set{y\in B_{\eps/4}}{\Gamma(x_0)-u(x_0+y)+\prodin{\xi}{y}\leq Cf(x_0)\eps^2}|
	\\
	\leq
	~&
	|\set{y\in B_{\eps/4}}{\Gamma(x_0+y)-u(x_0+y)\leq Cf(x_0)\eps^2}|
	\\
	=
	~&
	\big|B_{\eps/4}(x_0) \cap \{\Gamma-u\leq Cf(x_0)\eps^2\}\big|,
\end{split}
\end{equation*}
so the proof is finished.
\end{proof}

We obtain the same estimate in each cube $Q\in\mathcal{Q}_\eps(K_u)$ immediately as a corollary of the previous lemma.

\begin{corollary}\label{estimate Q}
Under the assumptions of Theorem~\ref{eps-ABP}, there exists $c>0$ such that
\begin{equation*}
	\big|3\sqrt{N}\,Q \cap \{\Gamma-u\leq C(\sup_Qf)\eps^2\}\big|
	\geq
	c|Q|
\end{equation*}
for each $Q\in\mathcal{Q}_\eps(K_u)$.
\end{corollary}

\begin{proof}
Let $Q\in\mathcal{Q}_\eps(K_u)$. Then there is $x_0\in \overline{Q}\cap K_u$. On the other hand,  since $\diam Q=\eps/4$, if we denote by $x_Q$ the center of $Q$, we get that $|x_Q-x_0|\leq\diam Q/2$ and
\begin{equation*}
	B_{\eps/4}(x_0)
	=
	B_{\diam Q}(x_0)
	\subset
	B_{\frac{3}{2}\diam Q}(x_Q)
	\subset
	3\sqrt{N}\,Q.
\end{equation*}
Hence, by Lemma~\ref{estimate B_eps}, using this inclusion and recalling that $\eps^N=(4\sqrt{N})^N|Q|$ we complete the proof.
\end{proof}

\subsection{A barrier function for $\L_\eps^-$}

Another ingredient needed in  the proof of the measure estimate Lemma~\ref{first} is a construction of a barrier for the minimal Pucci-type operator defined in \eqref{L-eps-}. To that end, we prove the following technical inequality for real numbers.

\begin{lemma}
Let $\sigma>0$. If $a,b>0$ and $c\in\R$ such that $|c|<a+b$ then
\begin{multline}\label{ineq:abc}
	(a+b+c)^{-\sigma}+(a+b-c)^{-\sigma}-2a^{-\sigma}
	\\
	\geq
	2\sigma a^{-\sigma-1}\left[-b+\frac{\sigma+1}{2}\left(1-(\sigma+2)\frac{b}{a}\right)\frac{c^2}{a}\right].
\end{multline}
\end{lemma}

%%%%%%%%%%%%%%%%%%%%%%%%

\begin{proof}
The inequality
\begin{equation*}
	(t+h)^{-\sigma}+(t-h)^{-\sigma}-2t^{-\sigma}
	\geq
	\sigma(\sigma+1)t^{-\sigma-2}h^2
\end{equation*}
holds for every $0<|h|<t$.
This can be seen by considering the Taylor expansion in $h$ of the LHS with error of order 4 and bound the error since it is positive.

 Then replacing $t=a+b$ and $h=c$ we obtain that
\begin{equation*}
	(a+b+c)^{-\sigma}+(a+b-c)^{-\sigma}
	\geq
	2(a+b)^{-\sigma}+\sigma(\sigma+1)(a+b)^{-\sigma-2}c^2.
\end{equation*}
Moreover, by using convexity we can estimate
\begin{equation*}
	(a+b)^{-\sigma}
	\geq
	a^{-\sigma}-\sigma a^{-\sigma-1}b
	=
	a^{-\sigma}\left(1-\sigma\frac{b}{a}\right),
\end{equation*}
and similarly
\begin{equation*}
	(a+b)^{-\sigma-2}
	\geq
	a^{-\sigma-2}\left(1-(\sigma+2)\frac{b}{a}\right).
\end{equation*}

Using these inequalities and rearranging terms we get
\begin{multline*}
	(a+b+c)^{-\sigma}+(a+b-c)^{-\sigma}-2a^{-\sigma}
	\\
\begin{split}
	\geq
	~&
	2a^{-\sigma}\left(1-\sigma\frac{b}{a}\right)+\sigma(\sigma+1)a^{-\sigma-2}\left(1-(\sigma+2)\frac{b}{a}\right)c^2-2a^{-\sigma}
	\\
	=
	~&
	2\sigma a^{-\sigma-1}\left[-b+\frac{\sigma+1}{2}\left(1-(\sigma+2)\frac{b}{a}\right)\frac{c^2}{a}\right],
\end{split}
\end{multline*}
and the proof is concluded.
\end{proof}

Next we construct a suitable barrier function. The importance of this function, which will be clarified later, lies in the fact that, when added to a a subsolution $u$, its shape ensures that the contact set is localized in a fixed neighborhood of the origin. Recall the  notation $\L_\eps^-$ from \eqref{L-eps-}.

\begin{lemma}\label{barrier}
There exists a smooth function $\Psi:\R^N\to\R$ and $\eps_0>0$ such that
\begin{equation*}
	\begin{cases}
	\L_\eps^-\Psi+\psi\geq 0 & \text{ in } \R^N,
	\\
	\Psi\geq 2 & \text{ in } Q_3,
	\\
	\Psi\leq 0 & \text{ in } \R^N\setminus B_{2\sqrt{N}},
	\end{cases}
\end{equation*}
for every $0<\eps\leq\eps_0$, where $\psi:\R^N\to\R$ is a smooth function such that 
\begin{equation*}
	\psi\leq\psi(0) \text{ in } \R^N
	\qquad\text{ and }\qquad
	\psi\leq 0 \text{ in } \R^N\setminus B_{1/4}.
\end{equation*}
\end{lemma}

\begin{proof}
The proof is constructive. Let $\sigma>0$ to be fixed later and define
\begin{equation*}
	\Psi(x)
	=
	A(1+|x|^2)^{-\sigma}-B
\end{equation*}
for each $x\in\R^N$, where $A,B>0$ are chosen such that
\begin{equation*}
	\Psi(x)
	=
	\begin{cases}
	2 & \text{ if } |x|=\frac{3}{2}\sqrt{N},
	\\
	0 & \text{ if } |x|=2\sqrt{N}.
	\end{cases}
\end{equation*}
Then $\Psi\leq 0$ in $\R^N\setminus B_{2\sqrt{N}}$ and $\Psi\geq 2$ in $Q_3\subset B_{3/2\sqrt{N}}$. We show that $\Psi$ satisfies the remaining condition for a suitable choice of the exponent $\sigma$ independently of $\eps$.

Since $\Psi$ is radial, we can assume without loss of generality that $x=(|x|,0,\ldots,0)$. Then
\begin{equation*}
	\Psi(x+\eps y)
	=
	A(1+|x+\eps y|^2)^{-\sigma}-B
	=
	A(1+|x|^2+\eps^2|y|^2+2\eps|x|y_1)^{-\sigma}-B
\end{equation*}
for every $y\in\R^N$. Thus, recalling \eqref{ineq:abc} with $a=1+|x|^2$, $b=\eps^2|y|^2$ and $c=2\eps|x|y_1$ we obtain that
\begin{equation*}
\begin{split}
	\delta\Psi(x,\eps y)
	&
	=
	\Psi(x+\eps y)+\Psi(x-\eps y)-2\Psi(x)
	\\
	&
	\geq
	2\eps^2A\sigma (1+|x|^2)^{-\sigma-1}\left[-|y|^2+2(\sigma+1)\left(1-(\sigma+2)\frac{\eps^2|y|^2}{1+|x|^2}\right)\frac{|x|^2}{1+|x|^2}y_1^2\right]
	\\
	&
	\geq
	2\eps^2A\sigma (1+|x|^2)^{-\sigma-1}\left[-\Lambda^2+2(\sigma+1)(1-(\sigma+2)\Lambda^2\eps^2)\frac{|x|^2}{1+|x|^2}y_1^2\right]
\end{split}
\end{equation*}
for every $|y|<\Lambda$.

Fix $\eps_0=\eps_0(\Lambda,\sigma)$ such that
\begin{equation*}
	\eps_0
	\leq
	\frac{1}{\Lambda\sqrt{2(\sigma+2)}},
\end{equation*}
so
\begin{equation*}
	\delta\Psi(x,\eps y)
	\geq
	2\eps^2A\sigma (1+|x|^2)^{-\sigma-1}\left[-\Lambda^2+(\sigma+1)\frac{|x|^2}{1+|x|^2}y_1^2\right]
\end{equation*}
for every $|y|<\Lambda$ and $0<\eps\leq\eps_0$.
In consequence we can estimate
\begin{equation*}
	\inf_{z\in B_\Lambda}\delta\Psi(x,\eps z)
	\geq
	2\eps^2A\sigma (1+|x|^2)^{-\sigma-1}\left[-\Lambda^2\right]
\end{equation*}
and
\begin{equation*}
	\vint_{B_1}\delta\Psi(x,\eps y)\,dy
	\geq
	2\eps^2A\sigma (1+|x|^2)^{-\sigma-1}\left[-\Lambda^2+\frac{\sigma+1}{N+2}\cdot\frac{|x|^2}{1+|x|^2}\right],
\end{equation*}
where we have used that $\vint_{B_1}y_1^2\,dy=\frac{1}{N+2}$.
Replacing these inequalities in the definition of $\L_\eps^-\Psi(x)$, \eqref{L-eps-}, we obtain
\begin{equation*}
	\L_\eps^-\Psi(x)
	\geq
	A\sigma (1+|x|^2)^{-\sigma-1}\left[-\Lambda^2+\beta\frac{\sigma+1}{N+2}\cdot\frac{|x|^2}{1+|x|^2}\right]
	=\,:
	-\psi(x)
\end{equation*}
for every $x\in\R^N$ and $0<\eps\leq\eps_0$. It is easy to check that $\psi(x)\leq\psi(0)=A\sigma\Lambda^2$ for every $x\in\R^N$. Moreover
\begin{equation*}
	\psi(x)
	\leq
	A\sigma (1+|x|^2)^{-\sigma-1}\left[\Lambda^2-\frac{\beta(\sigma+1)}{17(N+2)}\right]
\end{equation*}
for every $|x|\geq 1/4$. Choosing large enough $\sigma=\sigma(N,\Lambda,\beta)>0$ we get that $\psi(x)\leq 0$ for every $|x|\geq 1/4$ and the proof is finished.
\end{proof}

\subsection{Estimate for the distribution function of $u$}

In the next lemma we adapt \cite[Lemma 10.1]{caffarellis09} to pass from a pointwise estimate to an estimate in measure. This is done by combining the estimate for the difference between $u$ and $\Gamma$ near the contact set with the $\eps$-ABP estimate.

\begin{lemma}
\label{first}
There exist $\eps_0,\rho>0$, $M\geq 1$ and $0<\mu<1$ such that if $u$ is a bounded measurable function satisfying
\begin{equation*}
	\begin{cases}
	\L_\eps^-u\leq\rho & \text{ in } B_{2\sqrt{N}},
	\\
	u\geq 0 & \text{ in } \R^N,
	\end{cases}
\end{equation*}
for some $0<\eps\leq\eps_0$ and
\begin{equation*}
	\inf_{Q_3}u
	\leq
	1,
\end{equation*}
then
\begin{equation*}
	|\{u>  M\}\cap Q_1|
	\le 
	\mu.
\end{equation*}
\end{lemma}
%%%%%%%%%%%%%%%%%%%%%%%%%%%%%
\begin{proof}
The idea of the proof is as follows: first we use the auxiliary functions $\Psi$ and $\psi$ from Lemma~\ref{barrier} to define a new function 
$$
v=\Psi-u,
$$
which satisfies the assumptions in Theorem~\ref{eps-ABP} ($\eps$-ABP estimate) with $f=\psi+\rho$. Then we use the $\eps$-ABP together with the pointwise estimate $\inf_{Q_3}u\leq 1$ and the negativity of $\psi$ outside $B_{1/4}$ to obtain a lower bound for the measure of the union of all cubes $Q\in\mathcal{Q}_\eps(K_v\cap B_{1/4})$. Combining this with the estimate of the difference between $v$ and its concave envelope at each cube $Q$ (Corollary~\ref{estimate Q}) we can deduce the desired measure estimate for $u$.

Let $v=\Psi-u$ where $\Psi$ is the function from Lemma~\ref{barrier}. Since $u\geq 0$ and $\Psi\leq 0$ in $\R^N\setminus B_{2\sqrt{N}}$, then $v\leq 0$ in $\R^N\setminus B_{2\sqrt{N}}$. On the other hand,
\begin{equation*}
	\sup_{Q_3}v
	\geq
	\inf_{Q_3}\Psi-\inf_{Q_3}u
	\geq
	1.
\end{equation*}
Similarly, since $\delta v(x,\eps y)=\delta\Psi(x,\eps y)-\delta u(x,\eps y)$, then
\begin{equation*}
	\sup_{z\in B_\Lambda}\delta v(x,\eps z)
	\geq
	\inf_{z\in B_\Lambda}\delta \Psi(x,\eps z)-\inf_{z\in B_\Lambda}\delta u(x,\eps z)\end{equation*}
so we have that
\begin{equation*}
	\L_\eps^+v(x)
	\geq
	\L_\eps^-\Psi(x)-\L_\eps^-u(x)
	\geq
	-\psi(x)-\rho.
\end{equation*}

Summarizing, $v=\Psi-u$ satisfies $\sup_{Q_3}v\geq 1$ and
\begin{equation*}
	\begin{cases}
	\L_\eps^+v+\psi+\rho\geq 0 & \text{ in } B_{2\sqrt{N}},
	\\
	v\leq 0 & \text{ in } \R^N\setminus B_{2\sqrt{N}}.
	\end{cases}
\end{equation*}
Moreover, since $\psi$ is continuous, we are under the hypothesis of the $\eps$-ABP estimate in Theorem~\ref{eps-ABP}, and thus the following estimate holds,
\begin{equation*}
	\sup_{B_{2\sqrt{N}}}v
	\leq
	C_1\bigg(\sum_{Q\in\mathcal{Q}_\eps(K_v)}(\sup_Q\psi^++\rho)^N|Q|\bigg)^{1/N},
\end{equation*}
where $C_1>0$. Then, since $Q_3\subset B_{2\sqrt{N}}$ and $\sup_{Q_3}v\geq 1$, we obtain
\begin{equation*}
\begin{split}
	\frac{1}{C_1}
	\leq
	~&
	\bigg(\sum_{Q\in\mathcal{Q}_\eps(K_v)}(\sup_Q\psi^++\rho)^N|Q|\bigg)^{1/N}
	\\
	\leq
	~&
	\bigg(\sum_{Q\in\mathcal{Q}_\eps(K_v)}(\sup_Q\psi^+)^N|Q|\bigg)^{1/N}
	+
	\rho\bigg(\sum_{Q\in\mathcal{Q}_\eps(K_v)}|Q|\bigg)^{1/N},
\end{split}
\end{equation*}
where the second inequality follows immediately from Minkowski's inequality. Since $K_v\subset B_{2\sqrt{N}}$ and $\diam Q=\eps/4$ for each $Q\in\mathcal{Q}_\eps(K_v)$ then
\begin{equation*}
	\sum_{Q\in\mathcal{Q}_\eps(K_v)}|Q|
	\leq
	|B_{2\sqrt{N}+\eps/4}|
	\leq
	C_2^N,
\end{equation*}
for every $0<\eps\leq\eps_0$.
Replacing in the previous estimate and rearranging terms we get
\begin{equation*}
	\frac{1}{C_1}-C_2\rho
	\leq
	\bigg(\sum_{Q\in\mathcal{Q}_\eps(K_v)}(\sup_Q\psi^+)^N|Q|\bigg)^{1/N}.
\end{equation*}
Choosing small enough $\rho>0$ 
we have that 
\begin{equation*}
	\frac{1}{(2C_1)^N}
	\leq
	\sum_{Q\in\mathcal{Q}_\eps(K_v)}(\sup_Q\psi^+)^N|Q|.
\end{equation*}
Next we observe that by Lemma~\ref{barrier}, $\psi\leq 0$ in $\R^N\setminus B_{1/4}$, so $\psi^+\equiv 0$ for each $Q\in\mathcal{Q}_\eps(K_v)$ such that $Q\cap B_{1/4}=\emptyset$, while we estimate $\sup_Q\psi^+\leq\psi(0)$ when $Q\cap B_{1/4}\neq\emptyset$. Thus
\begin{equation*}
	\frac{1}{(2C_1\psi(0))^N}
	\leq
	\sum_{Q\in\mathcal{Q}_\eps(K_v\cap B_{1/4})}|Q|,
\end{equation*}
and recalling Corollary~\ref{estimate Q}, we obtain the following inequality,
\begin{equation*}
	\frac{c}{(2C_1\psi(0))^N}
	\leq
	\sum_{Q\in\mathcal{Q}_\eps(K_v\cap B_{1/4})}\big|3\sqrt{N}\,Q \cap \{\Gamma-v\leq C(\sup_Q\psi^++\rho)\eps^2\}\big|.
\end{equation*}
Notice that $3\sqrt{N}\,Q\subset B_{1/2}\subset Q_1$ for each $Q\in\mathcal{Q}_\eps(K_v\cap B_{1/4})$ and every $0<\eps\leq\eps_0$ with $\eps_0>0$ sufficiently small, so
\begin{equation*}
	3\sqrt{N}\,Q \cap \{\Gamma-v\leq C(\sup_Q\psi^++\rho)\eps^2\}
	\subset
	Q_1 \cap \{\Gamma-v\leq C(\psi(0)+\rho)\eps^2\}
\end{equation*}
for each $Q\in\mathcal{Q}_\eps(K_v\cap B_{1/4})$, where the fact that $\sup_Q\psi^+\leq\psi(0)$ has been used again here. 
Furthermore, if $\ell=\ell(N)\in\N$ is the unique odd integer such that $\ell-2<3\sqrt{N}\leq\ell$, then each cube $Q\in\mathcal{Q}_\eps(K_v\cap B_{1/4})$ is contained in at most $\ell^N$cubes of the form $3\sqrt{N}\,Q'$ with $Q'\in\mathcal{Q}_\eps(K_v\cap B_{1/4})$, and in consequence 
\begin{equation*}
	\frac{c}{(2C_1\psi(0))^N}
	\leq
	\ell^N\big|Q_1 \cap \{\Gamma-v\leq C(\psi(0)+\rho)\eps^2\}\big|.
\end{equation*}

Finally, since $\Gamma\geq 0$, $v=\Psi-u\leq\Psi(0)-u$ and $\eps\leq\eps_0$,
\begin{equation*}
	\frac{c}{(2C_1\psi(0)\ell)^N}
	\leq
	\big|Q_1 \cap \{u\leq \Psi(0)+C(\psi(0)+\rho)\eps_0^2\}\big|.
\end{equation*}
Then let $M:\,=\Psi(0)+C(\psi(0)+\rho)\eps_0^2$ and $1-\mu:\,=c(2C_1\psi(0)\ell)^{-N}$, so that we get 
\begin{align*}
	1-\mu
	\leq
	\big|Q_1 \cap \{u\leq M\}\big|,
\end{align*}
which immediately implies the claim.
\end{proof}

\section{De Giorgi oscillation estimate}

A key intermediate result towards the oscillation estimate (Lemma \ref{DeGiorgi}), H\"older regularity (Theorem \ref{Holder}) and Harnack's inequality is a power decay estimate for $|\{u>t\}\cap Q_1|$. This will be Lemma~\ref{measure bound}. It is based on the measure
estimates Lemma~\ref{first} and Lemma~\ref{second}, as well as a discrete version of the Calder\'on-Zygmund decomposition, Lemma~\ref{CZ} below.

\subsection{Calder\'on-Zygmund decomposition}

The discrete nature of the DPP does not allow to apply the rescaling argument to arbitrary small dyadic cubes. To be more precise, since all the previous estimates require certain bound $\eps_0>0$ for the scale-size in the DPP, and since the extremal Pucci-type operators $\L_\eps^\pm$ rescale as $\L_{2^\ell\eps}^\pm$ in each dyadic cube of generation $\ell$, the rescaling argument will only work on those dyadic cubes of generation $\ell\in\N$ satisfying $2^\ell\eps<\eps_0$. For that reason, the dyadic splitting in  the Calder\'on-Zygmund decomposition has to be stopped at generation $L$, and in consequence the Calder\'on-Zygmund decomposition lemma has to be adapted. We need an additional criterion for selecting cubes in order to control the error caused by stopping the process at generation $L$. We use the idea from \cite{arroyobp}.

We use the following notation: $\mathcal D_\ell$ is the family of dyadic open subcubes of $Q_1$ of generation $\ell\in\N$, where $\mathcal D_0=\{Q_1\}$, $\mathcal D_1$ is the family of $2^N$ dyadic cubes obtained by dividing $Q_1$, and so on. Given $\ell\in\N$ and $Q\in\mathcal D_\ell$ we define $\mathrm{pre}(Q)\in\mathcal D_{\ell-1}$ as the unique dyadic cube in $\mathcal D_{\ell-1}$ containing $Q$.
 
\begin{lemma}[Calder\'on-Zygmund]\label{CZ}
Let $A\subset B\subset Q_1$ be measurable sets, $\delta_1,\delta_2\in (0,1)$ and $L\in\N$. Suppose that the following assumptions hold:
\begin{enumerate}
\item $|A|\leq\delta_1$;
\item \label{item:includedB}if $Q\in\mathcal D_\ell$ for some $\ell\leq L$ satisfies $|A\cap Q|>\delta_1|Q|$ then $\mathrm{pre}(Q)\subset B$;
\item \label{item:includedB2} if $Q\in\mathcal D_L$ satisfies $|A\cap Q|>\delta_2|Q|$ then $Q\subset B$;
\end{enumerate}
Then,
\begin{align*}
	|A|
	\leq
	\delta_1|B|+\delta_2.
\end{align*}
\end{lemma}
%%%%%%%%%%%%%%%%%%%
\begin{proof}
We will construct a collection of open cubes $\mathcal Q_B$, containing subcubes from generations $\mathcal D_0,\mathcal D_1,\dots,\mathcal D_L$. 
The cubes will be pairwise disjoint and will be contained in $B$.
Recall that by assumption 
$
|Q_1 \cap A|\leq \delta_1 \abs{Q_1}.
$
Then we split $Q_1$ into $2^N$ dyadic cubes $\mathcal D_1$. For those  dyadic cubes $Q\in \mathcal D_1$ that satisfy 
\begin{align}
\label{eq:treshold}
|A\cap Q|>\delta_1|Q|,
\end{align}
we select $\mathrm{pre}(Q)$ into $\mathcal Q_B$. 
Those cubes are included in $B$ because of assumption (\ref{item:includedB}).
 
For other dyadic cubes that do not satisfy \eqref{eq:treshold} and are not contained in any cube already included in $\mathcal Q_B$,
we keep splitting, and again repeat the selection according to \eqref{eq:treshold}.  We repeat splitting $L\in \mathbb N$ times. At the level $L$, in addition to the previous process, we also select those cubes $Q\in \mathcal D_L$ (not the predecessors) into $\mathcal Q_B$ for which 
\begin{align}
\label{eq:treshold2}
 |A\cap Q|> \delta_2 |Q|,
\end{align}
and are not contained in any cube already included in $\mathcal Q_B$.
Those cubes are included in $B$ because of assumption (\ref{item:includedB2}).

Observe that for $\mathrm{pre}(Q)$ selected according to \eqref{eq:treshold} into $\mathcal Q_B$, it holds that
\begin{align*}
|A\cap \mathrm{pre}(Q)|\le \delta_1|\mathrm{pre}(Q)|
\end{align*}
since otherwise we would have stopped splitting already at the earlier round. We also have $|A\cap Q|\le \delta_1|Q|$ for cubes $Q$ selected according to \eqref{eq:treshold2}  into $\mathcal Q_B$, since their predecessors were not selected according to \eqref{eq:treshold}. Summing up, for all the cubes $Q\in \mathcal Q_B$, it holds that
\begin{align}
\label{eq:meas-bound}
|A\cap Q|\le \delta_1|Q|.
\end{align}

Next we define $\mathcal G_L$ as a family of cubes of $\mathcal D_L$ that are not included in any of the cubes in $\mathcal Q_B$.
It immediately holds a.e.\ that
\[
A\subset Q_1=\bigcup_{Q\in\mathcal Q_B} Q  \cup \bigcup_{Q\in\mathcal G_L} Q.
\]
By this, using \eqref{eq:meas-bound} for every $Q\in \mathcal Q_B$, as well as observing that $|A\cap Q|\leq \delta_2|Q|$  by \eqref{eq:treshold2} for every $Q\in \mathcal G_L$, we get
\[
\begin{split}
|A|
&=\sum_{Q\in\mathcal Q_B} |A\cap Q| + \sum_{Q\in\mathcal G_L} |A\cap Q|\\
&\leq\sum_{Q\in\mathcal Q_B} \delta_1|Q| + \sum_{Q\in\mathcal G_L} \delta_2|Q|\\
&\leq  \delta_1 |B|+\delta_2 .
\end{split}
\]
In the last inequality, we used that the cubes in $\mathcal Q_B$ are included in $B$, as well as the fact that they are disjoint by construction.
\end{proof}
%%%%%%%%%%%%%%%%%%%%%%%%

As we have already pointed out, we use the estimate from Lemma~\ref{first} to show that the condition (\ref{item:includedB}) in the  Calder\'on-Zygmund lemma is satisfied. To ensure that the remaining condition is satisfied for the dyadic cubes in $\mathcal{D}_L$ not considered before stopping the dyadic decomposition, we prove the following result using the equation. Here $\eps$ is `relatively large'.

\begin{lemma}
\label{second}
Let $0<\eps_0<1$ and $\rho>0$. Suppose that $u$ is a bounded measurable function satisfying
\begin{equation*}
	\begin{cases}
	\L_\eps^-u\leq\rho & \text{ in } Q_{10\sqrt{N}},
	\\
	u\geq 0 & \text{ in } \R^N,
	\end{cases}
\end{equation*}
for some $\frac{\eps_0}{2}\leq\eps\leq\eps_0$. There exists a constant $c=c(\eps_0,\rho)>0$ such that if
\begin{equation*}
	|\{u> K\}\cap Q_1|
	>
	\frac{c}{K}
\end{equation*}
holds for some $K>0$, then
\begin{equation*}
	u> 1 \quad \text{ in }Q_1.
\end{equation*}
\end{lemma}
%%%%%%%%%%%%%%%%%%%%%%%%%%%
\begin{proof}
By the definition of the minimal Pucci-type operator $\L_\eps^-$ and since $\L_\eps^- u(x)\leq \rho$ for every $x\in Q_{10\sqrt{N}}$ by assumption, rearranging terms we have
\[
\begin{split}
	u (x) 
	&
	\geq
	\alpha\inf_{\nu\in \M(B_\Lambda)}  \int u(x+\eps v) \,d\nu (v)+\beta\vint_{B_\eps(x)} u(y)\,dy-\eps^2\rho
	\\
	&\geq
	\beta\vint_{B_\eps(x)} u(y)\,dy-\eps^2\rho,
\end{split}
\]
where in the second inequality we have used that $u\geq 0$ to estimate the $\alpha$-term by zero.
Then, by considering $f=\frac{\chi_{B_1}}{|B_1|}$, we can rewrite this inequality as
\begin{equation*}
	u (x) 
	\geq
	\frac{\beta}{\eps^N}\int f\Big(\frac{y-x}{\eps}\Big) u(y)\,dy-\eps^2\rho,
\end{equation*}
which holds for every $x\in Q_{10\sqrt{N}}$, and in particular for every $|x|<5\sqrt{N}$. Next observe that if $|x|+\eps<5\sqrt{N}$, then $y\in Q_{10\sqrt{N}}$ for every $y\in B_\eps(x)$, and thus applying twice the previous inequality we can estimate by using change of variables
\begin{equation*}
\begin{split}
	u (x)
	\geq
	~&
	\frac{\beta}{\eps^N}\int f\Big(\frac{y-x}{\eps}\Big) \left(\frac{\beta}{\eps^N}\int f\Big(\frac{z-y}{\eps}\Big) u(z)\,dz-\eps^2\rho\right)\,dy-\eps^2\rho
	\\
	=
	~&
	\frac{\beta^2}{\eps^N}\int \left(\frac{1}{\eps^N}\int f\Big(\frac{y-x}{\eps}\Big) f\Big(\frac{z-y}{\eps}\Big) \,dy\right)u(z)\,dz-(1+\beta)\eps^2\rho
	\\
	=
	~&
	\frac{\beta^2}{\eps^N}\int (f*f)\Big(\frac{z-x}{\eps}\Big)u(z)\,dz-(1+\beta)\eps^2\rho,
\end{split}
\end{equation*}
%%%%%%%%%%%%%%%%%%%
which holds for every $|x|<5\sqrt{N}-\eps$.

Let $n\in\N$ to be fixed later and assume that $|x|+(n-1)\eps<5\sqrt{N}$. By iterating this argument $n$ times we obtain
\begin{equation}\label{inequality}
\begin{split}
	u(x)
	\geq
	~&
	\frac{\beta^n}{\eps^N}\int f^{*n}\Big(\frac{y-x}{\eps}\Big) u(y)\,dy-(1+\beta+\beta^2+\cdots+\beta^{n-1})\eps^2\rho
	\\
	\geq
	~&
	\frac{\beta^n}{\eps^N}\int f^{*n}\Big(\frac{y-x}{\eps}\Big) u(y)\,dy-\frac{\eps^2\rho}{1-\beta}
\end{split}
\end{equation}
for every $|x|<5\sqrt{N}-(n-1)\eps$, where $f^{*n}$ denotes the convolution of $f$ with itself $n$ times. Observe that $f^{*n}$ is a radial decreasing function and $f^{*n}>0$ in $B_n$. Thus, since $\eps\geq\frac{\eps_0}{2}$ by assumption,
\begin{equation*}
	f^{*n}\Big(\frac{y-x}{\eps}\Big)
	\geq
	f^{*n}\Big(\frac{2(y-x)}{\eps_0}\Big),
\end{equation*}
which is strictly positive whenever $|y-x|<\frac{n\eps_0}{2}$.
Now fix $n\in\N$ such that $|x|<5\sqrt{N}-(n-1)\eps_0$ for every $x\in Q_1$ and $|y-x|<\frac{n\eps_0}{2}$ for every $x,y\in Q_1$, that is
$n\in\N$ such that
\begin{equation*}
	2\sqrt{N}
	<
	n\eps_0
	<
	\frac{9}{2}\sqrt{N}+\eps_0.
\end{equation*}
Then
\begin{equation*}
	f^{*n}\Big(\frac{y-x}{\eps}\Big)
	\geq
	f^{*n}\Big(\frac{2\sqrt{N}e_1}{\eps_0}\Big)
	=\,:
	C
	>
	0
\end{equation*}
for every $x,y\in Q_1$. In this way $Q_1$ is contained in the support of $y\mapsto f^{*n}\big(\frac{y-x}{\eps}\big)$ for every $x\in Q_1$, so recalling that $u\geq 0$ we can estimate
\begin{equation*}
\begin{split}
	\int f^{*n}\Big(\frac{y-x}{\eps}\Big) u(y)\,dy
	\geq
	~&
	\int_{Q_1} f^{*n}\Big(\frac{y-x}{\eps}\Big) u(y)\,dy
	\\
	\geq
	~&
	C\int_{Q_1}u(y)\,dy
	\\
	\geq
	~&
	C\int_{\{u> K\}\cap Q_1}u(y)\,dy
	\\
	>
	~&
	C|\{u> K\}\cap Q_1|\,K
\end{split}
\end{equation*}
for each $K>0$.
Replacing this in \eqref{inequality} and recalling that $\eps\leq\eps_0$ we get
\begin{equation*}
\begin{split}
	u(x)
	>
	~&
	C\frac{\beta^n}{\eps^N}|\{u> K\}\cap Q_1|\,K-\frac{\eps^2\rho}{1-\beta}
	\\
	\geq
	~&
	C\frac{\beta^n}{\eps_0^N}|\{u> K\}\cap Q_1|\,K-\frac{\eps_0^2\rho}{1-\beta}
\end{split}
\end{equation*}
for each $K>0$ and every $x\in Q_1$.

Finally, let us fix $c=\frac{\eps_0^N}{C\beta^n}\big(1+\frac{\eps_0^2\rho}{1-\beta}\big)$. By assumption, $|\{u> K\}\cap Q_1|\,K>c$ holds for some $K>0$, so
\begin{equation*}
\begin{split}
	u(x)
	>
	C\frac{\beta^n}{\eps_0^N}c-\frac{\eps_0^2\rho}{1-\beta}
	=
	1
\end{split}
\end{equation*}
for every $Q_1$ and the proof is finished.
\qedhere
\end{proof}

%%%%%%%%%%%%%%%%%%%%%%%%%%%%%%%%%%%%
\subsection{Power decay estimate}

The power decay estimate (Lemma~\ref{measure bound}) is obtained by deriving an estimate between the superlevel sets of $u$ and then iterating the estimate. In order to obtain the estimate between the superlevel sets, we use a discrete version of the  Calder\'on-Zygmund decomposition (Lemma~\ref{CZ}) together with  the preliminary measure estimates from Lemma~\ref{first} and Lemma~\ref{second}.

\begin{lemma}
\label{lem:main}
There exist $\eps_0,\rho,c>0$, $M\geq 1$ and $0<\mu<1$ such that if $u$ is a bounded measurable function satisfying
\begin{equation*}
	\begin{cases}
	\L_\eps^-u\leq\rho & \text{ in } Q_{10\sqrt{N}},
	\\
	u\geq 0 & \text{ in } \R^N,
	\end{cases}
\end{equation*}
for some $0<\eps\leq\eps_0$ and
\begin{equation*}
	\inf_{Q_3}u
	\leq
	1,
\end{equation*}
then
\begin{equation*}
	|\{u> K^k\}\cap Q_1|
	\leq
	\frac{c}{(1-\mu)K}+\mu^k,
\end{equation*}
holds for every $K\ge M$ and $k\in\N$.
\end{lemma}

\begin{proof}
The values of $M$, $\mu$, $\eps_0$ and $\rho$  are already given by Lemma~\ref{first}, while $c$ has been fixed in Lemma~\ref{second}. 

For $k=1$, by Lemma~\ref{first}, we have
\[
|\{u> K\}\cap Q_1|\le  |\{u> M\}\cap Q_1|\le\mu\le \frac{c}{K}+\mu.
\]
Now we proceed by induction. 
We consider
\[
A:=A_{k}:=\{u> K^k\}\cap Q_1 \quad \text{ and } B:=A_{k-1}:=\{u> K^{k-1}\}\cap Q_1.
\]

We have $A\subset B\subset Q_1$ and $|A|\le \mu$.
We apply Lemma~\ref{CZ} for $\delta_1=\mu$, $\delta_2=\frac{c}{K}$ and $L\in\N$ such that $2^L\eps<\eps_0\leq 2^{L+1}\eps$.
We have to check in two cases that certain dyadic cubes are included in $B$.

Observe that since $|A|\le \mu$, the first assumption in Lemma~\ref{CZ} is satisfied. Next we check that the remaining conditions in Lemma~\ref{CZ} are also satisfied. Given any cube $Q\in\mathcal{D}_\ell$ for some $\ell\leq L$, we define $\tilde u:Q_1\to\R$ as a rescaled version of $u$ restricted to $Q$, that is
\begin{equation}\label{tilde u}
	\tilde u(y)
	=
	\frac{1}{K^{k-1}}\,u(x_0+2^{-\ell}y)
\end{equation}
for every $y\in Q$, where $x_0$ stands for the center of $Q$. Then
\begin{equation*}
	|\{\tilde u> K\}\cap Q_1|
	=
	2^{N\ell}|\{u> K^k\}\cap Q|
	=
	\frac{|A\cap Q|}{|Q|}.
\end{equation*}

Let us suppose that $Q$ is a cube in $\mathcal{D}_\ell$ for some $\ell\leq L$ satisfying
\begin{align}
\label{eq:meas-assump}
|A\cap Q|>\mu |Q|.
\end{align}
We have to check that $\mathrm{pre}(Q)\subset B$. Let us suppose on the contrary that the inclusion does not hold, that is that there exists $\tilde x\in\mathrm{pre}(Q)$ such that $u(\tilde x)\le K^{k-1}$. By \eqref{tilde u} we have that
\begin{equation*}
	\delta\tilde u(y,\tilde\eps z)
	=
	\frac{1}{K^{k-1}}\,\delta u(x_0+2^{-\ell}y,\eps z),
\end{equation*}
where $\tilde\eps=2^\ell\eps\leq 2^L\eps<\eps_0$, and $\delta\tilde u(y,\tilde\eps z)$ is defined according to (\ref{eq:delta}). Replacing this in the definition of $\L_\eps^-$ in \eqref{L-eps-}, and since $\L_\eps^-u\leq\rho$ by assumption, we obtain
\begin{equation*}
	\L_{\tilde\eps}^-\tilde u(y)
	=
	\frac{1}{2^{2\ell}K^{k-1}}\,\L_\eps^-u(x_0+2^{-\ell}y)
	\leq
	\frac{\rho}{2^{2\ell}K^{k-1}}
	\leq
	\rho.
\end{equation*}
where we have used that $K\geq M\geq1$. Moreover $\tilde u\geq 0$ and $\inf_{Q_3}\tilde u\leq 1$ since  $u(\tilde x)\le K^{k-1}$ by the counter assumption.
Hence, the rescaled function $\tilde u$ satisfies the assumptions in Lemma~\ref{first}, and thus
\begin{equation*}
	\frac{|A\cap Q|}{|Q|}
	=
	|\{\tilde u> K\}\cap Q_1|
	\le
	\mu,
\end{equation*}
which contradicts (\ref{eq:meas-assump}). Thus $\mathrm{pre}(Q)\subset B$ and the second condition in Lemma~\ref{CZ} is satisfied.

Suppose now that $Q\in\mathcal{D}_L$ is a dyadic cube satisfying
\[
|A\cap Q|>\frac{c}{K}|Q|.
\]
Then
\begin{equation*}
	|\{\tilde u> K\}\cap Q_1|
	=
	\frac{|A\cap Q|}{|Q|}
	>
	\frac{c}{K},
\end{equation*}
and by Lemma~\ref{second} we have that $\tilde u\geq 1$ in $Q_1$. Recalling \eqref{tilde u} we get that $u\geq K^{k-1}$ in $Q$, and thus $Q\subset B$ as desired.

Finally, the assumptions in Lemma~\ref{CZ} are satisfied, so we can conclude that
\begin{equation*}
	|A|
	\leq
	\frac{c}{K}+\mu|B|,
\end{equation*}
so the result follows by induction.
We get
\begin{equation*}
	|\{u> K^k\}\cap Q_1|
	\leq
	\frac{c}{K}(1+\mu+\cdots+\mu^{k-1})+\mu^k
	\leq
	\frac{c}{(1-\mu)K}+\mu^k
\end{equation*}
as desired.
\end{proof}

Next we show that a convenient choice of the constants in the previous result immediately leads to the desired power decay estimate for $|\{u\geq t\}\cap Q_1|$.

\begin{lemma}
\label{measure bound}

Let $u$ be a function satisfying the conditions from Lemma~\ref{lem:main}. There exist $a>0$ and $d\geq 1$ such that
\[
|\{u>  t\}\cap Q_1|\leq d  e^{-\sqrt{\frac{\log t }{a}}}
\]
for every $t\ge 1$.
\end{lemma}

\begin{proof}
Let $M\geq 1$ and $\mu\in(0,1)$ be the constants from Lemma~\ref{lem:main}. Let us fix $a=\frac{1}{\log\frac{1}{\mu}}>0$. Then given $t\geq 1$ we choose $K=K(t)=e^{\sqrt{\log(t)/a}}\geq 1$, so $t=K^{a\log K}$. We distinguish two cases.

First, if $K=K(t)\geq M$, recalling Lemma~\ref{lem:main} we have that the estimate
\begin{equation*}
	|\{u>K^k\}\cap Q_1|
	\leq
	\frac{c}{(1-\mu)K}+\mu^k
\end{equation*}
holds for every $k\in\N$. In particular, if we fix $k=\lfloor a\log K\rfloor$ we get that
\begin{equation*}
	K^k
	\leq
	K^{a\log(K)}
	=
	t
\end{equation*}
and
\begin{equation*}
	\mu^k
	<
	\mu^{a\log(K)-1}
	=
	\frac{1}{K\mu}.
\end{equation*}
Using these inequalities together with the estimate from Lemma~\ref{lem:main} we obtain
\begin{equation*}
\begin{split}
	|\{u> t\}\cap Q_1|
	\leq
	~&
	|\{u> K^k\}\cap Q_1|
	\\
	\leq
	~&
	\frac{c}{(1-\mu) K}+\mu^k
	\\
	\leq
	~&
	\left(\frac{c}{1-\mu}+\frac{1}{\mu}\right)\frac{1}{K}
	\\
	=
	~&
	\left(\frac{c}{1-\mu}+\frac{1}{\mu}\right)e^{-\sqrt{\frac{\log t}{a}}},
\end{split}
\end{equation*}
where in the last equality we have used the definition of $K=K(t)$.

On the other hand, if $K(t)<M$ then we can roughly estimate
\begin{equation*}
\begin{split}
	|\{u>t\}\cap Q_1|
	\leq
	1
	<
	\frac{M}{K(t)}
	=
	Me^{-\sqrt{\frac{\log t}{a}}}.
\end{split}
\end{equation*}

Finally, choosing $d=\max\{M,\frac{c}{1-\mu}+\frac{1}{\mu}\}\geq 1$, the result follows for every $t\geq 1$.
\end{proof}

We prove here the De Giorgi oscillation lemma. The lemma follows from the measure estimate in a straightforward manner. Harnack's inequality requires an additional argument that we postpone to the next section.

\begin{lemma}[De Giorgi oscillation lemma]
\label{DeGiorgi}
Given $\theta\in (0,1)$, there exist $\eps_0,\rho>0$ and $\eta=\eta(\theta)\in (0,1)$ such that if $u$ satisfies 
\begin{equation*}
	\begin{cases}
		\L_\eps^-u\leq \eta\rho & \text{ in } Q_{10\sqrt{N}},
		\\
		u\geq 0 & \text{ in } \R^N, 
	\end{cases}
\end{equation*}
for some $0<\eps<\eps_0$
and
\[
|Q_{1}\cap \{u> 1\}|\geq \theta,
\]
then
\[
\inf_{Q_3} u \geq \eta.
\]
\end{lemma}

\begin{proof}
We take $\eps_0,\rho>0$ given by Lemma~\ref{lem:main}.
Let $m=\displaystyle\inf_{Q_3}u$ for simplicity and define $\tilde u$ the rescaled version of $u$ given by
\begin{equation*}
	\tilde u(x)
	=
	\frac{u(x)}{m}
\end{equation*}
for every $x\in\R^N$. Then $\displaystyle\inf_{Q_3}\tilde u\leq 1$ and, by assumption,
\begin{equation*}
	|\{\tilde u>\frac{1}{m}\}\cap Q_1|
	=
	|\{u>1\}\cap Q_1|
	\geq
	\theta.
\end{equation*}

Now suppose that $\L_\eps^-u\leq\eta\rho$ where $0<\eta\leq m$ is a constant to be chosen later. Then
\begin{equation*}
	\L_\eps^-\tilde u
	=
	\frac{\L_\eps^- u}{m}
	\leq
	\frac{\eta\rho}{m}
	\leq
	\rho,
\end{equation*}
and recalling Lemma~\ref{measure bound} with $\tilde u$ and $t=\frac{1}{m}\geq 1$ (observe that in the case $m\geq 1$ we immediately get the result) we obtain
\begin{equation*}
	\theta
	\leq
	|\{\tilde u>\frac{1}{m}\}\cap Q_1|
	\leq
	de^{-\sqrt{\frac{\log\frac{1}{m}}{a}}}.
\end{equation*}
Rearranging terms we get
\begin{equation*}
	\inf_{Q_3}u
	=
	m
	\geq
	e^{-a\left(\log\frac{d}{\theta}\right)^2},
\end{equation*}
so choosing $\eta=\eta(\theta)=e^{-a\left(\log\frac{d}{\theta}\right)^2}\in(0,1)$ we finish the proof.
\end{proof}

Now we are in a position to state the H\"older estimate. The proof after obtaining the De Giorgi oscillation estimate is exactly as in \cite{arroyobp}.
The statement of the De Giorgi oscillation lemma here is different from the one  there. For the sake of completeness we prove that the statement here implies the one in \cite{arroyobp}.

\begin{lemma}
	\label{DeGiorgi-old}
	There exist $k>1$ and $C,\eps_0>0$ such that
	for every $R>0$ and $\eps<\eps_0R$, if $\L_\eps^+u\geq-\rho$ in $B_{kR}$ with $u\leq M$ in $B_{kR}$ and
	\[
	|B_{R}\cap \{u\leq m\}|\geq \theta |B_R|,
	\]
	for some $\rho>0$, $\theta\in (0,1)$ and $m,M\in\R$, then there exist $\eta=\eta(\theta)>0$ such that
	\[
	\sup_{B_R} u \leq (1-\eta)M+\eta m+ C R^2\rho .
	\]
\end{lemma}

\begin{proof}
	We can assume that $M>m$, given $\gamma>0$ we define
	\[
	\tilde u(x)=\frac{M-u(2Rx)}{M-m}+\gamma
	\]
	in $B_{k/2}$. 
	For $k=10N$ since $Q_{10\sqrt N}\subset B_{k/2}$ we get that $\tilde u$ is defined in $Q_{10\sqrt N}$.
	Since $u \leq M$ we get $\tilde u\geq 0$.
	Also, since $u\leq m$ implies $\tilde u>1$ we get
	\[
	|Q_{1}\cap \{u> 1\}|
	\geq |B_{1/2}\cap \{u> 1\}|
	\geq \frac{|B_{R}\cap \{u\leq m\}|}{|B_R|}
	\geq \theta.
	\]
	For $\tilde \eps =\frac{\eps}{2R}<\eps_0$, since $\L_\eps^+u\geq-\rho$, we get $\L_{\tilde\eps}^-\tilde u\leq \frac{4 R^2 \rho}{M-m}$.
	Therefore, Lemma~\ref{DeGiorgi} implies that there exists $\tilde \rho>0$ and $\tilde\eta=\tilde\eta(\theta)\in (0,1)$ such that if $\frac{4 R^2 \rho}{M-m}<\tilde \rho\tilde\eta$ we get
	\[
	\inf_{Q_3} \tilde  u \geq \tilde\eta.
	\]
	Then,
	\[
	\sup_{Q_{6R}} u
	\leq M(1-\tilde\eta+\gamma) + m(\tilde\eta +\gamma).
	\]
	Since $B_R\subset Q_{6R}$ and this holds for every $\gamma>0$, we get
	\[
	\sup_{B_R} u
	\leq M(1-\tilde\eta) + m \tilde\eta. 
	\]
	Finally we take $\eta=\tilde \eta$ and $C=\frac{4}{\tilde \rho}$. 
	Thus, if $\frac{4 R^2 \rho}{M-m}<\tilde \rho \tilde \eta$ the result immediately follows from above. And if $4 R^2 \rho\geq \tilde \rho\tilde\eta(M-m)$ we have
	\[
	\begin{split}
	\sup_{B_R} u
	&\leq M\\
	&= (1-\tilde\eta)M+\tilde\eta m+ \tilde\eta(M-m)\\
	&\leq (1-\tilde\eta)M+\tilde\eta m+ \frac{4 R^2 \rho}{\tilde \rho}\\
	&= (1-\eta)M+\eta m+ C R^2\rho. \qedhere
	\end{split}
	\]
\end{proof}

As we already mentioned, the H\"older estimate follows as in \cite{arroyobp}.

\begin{theorem}
\label{Holder}
There exists $\eps_0>0$ such that if $u$ satisfies $\L_\eps^+ u\ge -\rho$ and $\L_\eps^- u\le  \rho$ in $B_{R}$  where $\eps<\eps_0R$, there exist $C,\gamma>0$  such that
\[
|u(x)-u(z)|\leq \frac{C}{R^\gamma}\left(\sup_{B_{R}}|u|+R^2\rho\right)\Big(|x-z|^\gamma+\eps^\gamma\Big)
\]
for every $x, z\in B_{R/2}$.
\end{theorem}

\section{Harnack's inequality}

In this section we obtain an `asymptotic Harnack's inequality'.
First, we prove Lemma~\ref{lemma:apuja} that gives sufficient conditions to obtain the result.
One of the conditions of the lemma follows from Theorem~\ref{Holder} so then our task is to prove the other condition.

Before proceeding to the proof of the asymptotic Harnack we observe that the classical Harnack's inequality does not hold.

\begin{example}
\label{example}
Fix $\eps\in(0,1)$.
We consider $\Om=B_2\subset \R^N$ and $A=\{(x,0,\dots,0)\in \Om: x\in \eps\N\}$.
We define $\nu:\Omega\to\M(B_1)$ as
\begin{equation*}
\begin{split}
	&
	\nu_x(E)
	=
	\frac{|E\cap B_1|}{|B_1|}
	\qquad\text{ for } x\notin A,
	\\
	&
	\nu_x
	=
	\frac{\delta_{e_1}+\delta_{-e_1}}{2}
	\qquad\text{ for } x\in A,
\end{split}
\end{equation*}
where $e_1=(1,0,\dots,0)$. 
Now we construct a solution to the DPP $\L_\eps u=0$ in $\Omega$, we assume $\alpha>0$.
We define
\begin{equation*}
	u(x)
	=
	\begin{cases}
	a_k & \text{ if } x=(k\eps,0,\ldots,0), \ k\in\N, \\
	1 & \text{ otherwise,}
	\end{cases}
\end{equation*}
where $a_1=a>0$ is arbitrary and the rest of the  $a_k$'s are fixed so that $\L_\eps u(k\eps,0,\ldots,0)=0$ for each $k\in\N$.
 Observe that if $x\notin A$ then $\delta u(x,\eps y)=0$ a.e. $y\in B_1$ and thus
\begin{equation*}
	\L_\eps u(x)
	=
	\frac{1}{2\eps^2}\vint_{B_1}\delta u(x,\eps y)\, dy
	=
	0.
\end{equation*}
Otherwise, for $x=(k\eps,0,\ldots,0)$ we get
\begin{equation*}
\begin{split}
	\L_\eps u(x)
	=
	~&
	\frac{1}{2\eps^2}\left(\alpha\,\delta u(x,\eps e_1)+\beta\vint_{B_1} \delta u(x,\eps y)\,dy\right)
	\\
	=
	~&
	\frac{1}{\eps^2}\left(\alpha\,\frac{a_{k+1}+a_{k-1}}{2}+\beta-a_k\right).
\end{split}
\end{equation*}
Thus for the DPP to hold we must have
\begin{equation*}
	a_k
	=
	1-\alpha+\alpha\,\frac{a_{k-1}+a_{k+1}}{2}
\end{equation*}
for $k\in\N$ where we are denoting $a_0=1$.
Clearly this determines the values of the whole sequence, we explicitly calculate it.
Let $\varphi$ and $\bar \varphi$ be the solutions to the equation $x=\frac{\alpha}{2}(1+x^2)$, that is
\begin{equation*}
	\varphi
	=
	\frac{1+\sqrt{1-\alpha^2}}{\alpha}
	\quad\text{and}\quad
	\bar\varphi
	=
	\frac{1-\sqrt{1-\alpha^2}}{\alpha}.
\end{equation*}
Then 
\[
a_k=1+a\frac{\varphi^k-\bar\varphi^k}{\varphi-\bar\varphi}.
\]
Observe that $\inf_{B_1}u=1$ but $\sup_{B_1}u\geq a_1=  a$, so the Harnack inequality does not hold.
\end{example}

Let us observe that this does not contradict the H\"older estimate since $\sup_{B_{2}}|u|$ is large compared to $a$.

We begin the proof of the asymptotic Harnack inequality with the following lemma that gives sufficient conditions to obtain the result.
The lemma is a modification of Lemma 4.1 and Theorem 5.2 in \cite{luirops13}.
Our result, however, differs from the one there since, as observed above, in the present setting the classical Harnack's inequality does not hold.
The condition (ii) in Lemma 5.1 of \cite{luirops13} requires an estimate at level $\eps$ that we do not require here. 
Indeed, Example~\ref{example} shows that this condition does not necessarily hold in our setting.

\begin{lemma}\label{lemma:apuja}
Assume that $u$ is a positive function defined in $B_3\subset\R^n$ and there is $C\geq 1$, $\rho\geq0$ and $\eps>0$ such that
\begin{enumerate}
\item \label{item:for-harnack}  for some $\kappa,\lambda>0$,
\[
\inf_{B_r(x)}u\leq C\left(r^{-\lambda}\inf_{B_1}u+\rho\right)
\]
for every ${|x|\leq 2}$ and  $r\in (\kappa\varepsilon, 1)$,
\item  for some $\gamma>0$,
\label{item:holder} 
\begin{align*}
{\rm osc}\, (u,B_r(x))\leq C\left(\frac{r}{R}\right)^{\gamma} \left(\sup_{B_R(x)} u +R^2\rho\right)
\end{align*}
for every $|x|\leq 2$, $R\leq 1$ and $\varepsilon<r\leq\delta R$ with $\eps\kappa<R\delta$ where $\delta=(2^{1+2\lambda}C)^{-1/\gamma}$.

\end{enumerate}

Then 
\begin{equation*}
	\sup_{B_1}u
	\leq
	\tilde C\left(\inf_{B_1}u+\rho+\eps^{2\lambda}\sup_{B_3}u\right)
\end{equation*}
where $\tilde C=\tilde C(\kappa,\lambda,\gamma, C)=(2^{1+2\lambda}C)^{2\lambda/\gamma}\max(C2^{2+2\lambda},(2\kappa)^{2\lambda})$.
\end{lemma}

%%%%%%%%%%%%%%%%%%%%%%%%%%%%%%
\begin{proof}
We define $R_k=2^{1-k}$ and $M_k=4C(2^{-k}\delta)^{-2\lambda}$ for each $k=1,\ldots,k_0$, where $k_0=k_0(\eps)\in\N$ is fixed so that
\begin{equation*}
	2^{-(k_0+1)}
	\leq
	\frac{\kappa\eps}{2\delta}
	<
	2^{-k_0}.
\end{equation*}
Then
\begin{equation*}
	\eps^{2\lambda}
	\geq
	\left(\frac{\delta}{2\kappa}\right)^{2\lambda}\frac{M_1}{M_{k_0}}
\end{equation*}
and $\delta R_k\geq \delta R_{k_0}>\kappa\eps$.

We assume, for the sake of contradiction, that
\begin{equation*}
	\sup_{B_{1}}u
	>
	\tilde C\left(\inf_{B_1}u+\rho+\eps^{2\lambda}\sup_{B_3}u\right)
\end{equation*}
with
\begin{equation*}
	\tilde C
	=
	\max\left\{M_1,\left(\frac{2\kappa}{\delta}\right)^{2\lambda}\right\}.
\end{equation*}
We get
\begin{equation*}
	\sup_{B_1}u
	>
	M_1\left(\frac{1}{M_{k_0}}\sup_{B_3}u+\inf_{B_1}u+\rho\right).
\end{equation*}

We define $x_1=0$ and $x_2\in B_{R_1}(x_1)=B_1(0)$ such that 
\[
u(x_2)>M_1\left(\frac{1}{M_{k_0}}  \sup_{B_3}u + \inf_{B_1}u+\rho\right).
\]
We claim that we can construct a sequence $x_{k+1}\in B_{R_k}(x_k)$ such that 
\[
u(x_{k+1})>M_k\left(\frac{1}{M_{k_0}}  \sup_{B_3}u + \inf_{B_1}u+\rho\right).
\]
for $k=1,\dots,k_0$.

We proceed to prove this by induction, we fix $k$ and assume the hipotesis for the smaller values.
Since $\delta< 1$ we have $B_{\delta R_k}(x_k)\subset B_{R_k}(x_k)$.
Observe that $|x_k|\leq R_1+\cdots+R_{k-1}\leq 2$ and $1>\delta R_k>\kappa\eps$.
Then, by hypothesis (1) we get
\[
\begin{split}
\sup_{B_{R_{k}}(x_k)} u
&\geq C^{-1} \delta^{-\gamma} \left(\sup_{B_{\delta R_k}(x_k)}u - \inf_{B_{\delta R_k}(x_k)}u \right)-R_k^2\rho\\
&\geq C^{-1} \delta^{-\gamma} \left(u(x_k) - \inf_{B_{\delta R_k}(x_k)}u -C \delta^{\gamma}\rho\right).\\
\end{split}
\]
We apply hypothesis (2) for $B_{\delta R_k}(x_k)$, we get
\[
\begin{split}
\inf_{B_{\delta R_k}(x_k)}u+C \delta^{\gamma}\rho
&\leq  C(\delta R_k)^{-\lambda}\inf_{B_1}u+C\rho+C \delta^{\gamma}\rho\\
&<  2C(\delta R_k)^{-2\lambda}\inf_{B_1}u+\frac{M_{k-1}}{2}\rho\\
&= \frac{M_{k-1}}{2}\left(\inf_{B_1}u+\rho\right)\\
&< u(x_k)/2,\\
\end{split}
\]
where we have used that $C(1+\delta^\gamma)\leq 2C\leq M_1/2\leq M_{k-1}/2$
and the inductive hypothesis. 

Combining the last two inequalities we get
\[
\begin{split}
\sup_{B_{R_{k}}(x_k)} u
&> C^{-1} \delta^{-\gamma} \left(u(x_k) - u(x_k)/2\right)\\
&=C^{-1} \delta^{-\gamma} u(x_k)/2\\
&>C^{-1} \delta^{-\gamma} M_{k-1}/2\left(\frac{1}{M_{k_0}}  \sup_{B_3}u + \inf_{B_1}u+\rho\right)\\
&=M_k\left(\frac{1}{M_{k_0}}  \sup_{B_3}u + \inf_{B_1}u+\rho\right),
\end{split}
\]
where the last equality holds by the choice of $\delta$.
Then, we can choose $x_{k+1}\in B_{R_k}(x_k)$ such that 
\[
u(x_{k+1})>M_k\left(\frac{1}{M_{k_0}}  \sup_{B_3}u + \inf_{B_1}u+\rho\right).
\]

Therefore we get 
\[
u(x_{k_0+1})
>
\sup_{B_3}u + M_{k_0}\left( \inf_{B_1}u+\rho\right),
\]
which is a contradiction since $x_{k_0+1}\in B_2$.
\end{proof}

%%%%%%%%%%%%%%%%%%%%%%%%%%%%%%
So, now our task is to prove that solutions to the DPP satisfy the hypothesis of the previous lemma. We start working towards condition (\ref{item:for-harnack}).

\begin{theorem}\label{thm.cond2}
There exists $C,\sigma,\eps_0>0$ such that if $u$ is a bounded measurable function satisfying
\begin{equation*}
	\begin{cases}
	\L_\eps^-u\leq 0 & \text{ in } B_7,
	\\
	u\geq 0 & \text{ in } \R^N,
	\end{cases}
\end{equation*}
for some $0<\eps\leq\eps_0$, then
\begin{equation*}
	\inf_{B_r(z)}u
	\leq
	Cr^{-2\sigma}\inf_{B_1}u
\end{equation*}
for every $z\in B_2$ and $r\in(\kappa\eps,1)$, where $\kappa=\Lambda\sqrt{2(\sigma+1)}$.
\end{theorem}

\begin{proof}
Let $\Omega=B_4(z)\setminus\overline{B_r(z)}$. Our aim is to construct a subsolution $\Psi$ in the $\Lambda\eps$-neighborhood of $\Omega$, ie. in $\widetilde\Omega=B_{4+\Lambda\eps}(z)\setminus \overline{B_{r-\Lambda\eps}(z)}$, such that $\Psi\leq u$ in $\widetilde\Omega$.

Let $\Psi:\R^N\setminus\{0\}\to\R$ be the smooth function defined by
\begin{equation*}
	\Psi(x)
	=
	A|x-z|^{-2\sigma}-B
\end{equation*}
for certain $A,B,\sigma>0$, which is a radially decreasing function. The constants $A$ and $B$ are fixed in such a way that $\Psi\leq u$ in $\widetilde\Omega\setminus\Omega$, that is both in $\overline{B_r(z)}\setminus\overline{B_{r-\Lambda\eps}(z)}$ and $B_{4+\Lambda\eps}(z)\setminus B_4(z)$. More precisely, requiring
\begin{equation*}
	\Psi\big|_{\partial B_{r-\Lambda\eps}(z)}
	=
	\inf_{B_r(z)}u
	\qquad\text{ and }\qquad
	\Psi\big|_{\partial B_4(z)}
	=
	0,
\end{equation*}
and since $\Psi$ is radially decreasing, we obtain that $\Psi\leq u$ in $\widetilde\Omega\setminus\Omega$. Therefore these conditions determine $A$ and $B$ so that
\begin{equation*}
	\Psi(x)
	=
	\frac{|x-z|^{-2\sigma}-4^{-2\sigma}}{(r-\Lambda \eps)^{-2\sigma}-4^{-2\sigma}}\inf_{B_r}u.
\end{equation*}
Let us assume for the moment that $z=0$ and $x=(|x|,0\ldots,0)$. Similarly as in the proof of Lemma~\ref{barrier}, using \eqref{ineq:abc} we can estimate
\begin{equation*}
	\delta\Psi(x,\eps y)
	\geq
	2\eps^2 A\sigma|x|^{-2\sigma-2}\left[-\Lambda^2+2(\sigma+1)\left(1-(\sigma+2)\frac{\Lambda^2\eps^2}{r^2}\right)y_1^2\right]
\end{equation*}
for every $|x|>r>\Lambda\eps$ and $|y|<\Lambda$ (so that $|x+\eps y|>0$ and thus $\delta\Psi(x,\eps y)$ is well defined). Moreover, since $r\in(\kappa\eps,1)$ we get
\begin{equation*}
	1-(\sigma+2)\frac{\Lambda^2\eps^2}{r^2}
	\geq
	1-(\sigma+2)\frac{\Lambda^2}{\kappa^2}
	=
	\frac{1}{2},
\end{equation*}
where the equality holds for
\begin{equation*}
	\kappa
	=
	\Lambda\sqrt{2(\sigma+2)}
	\geq
	2\Lambda.
\end{equation*}
This also sets out an upper bound for $\eps$: the inequality $\kappa\eps<1$ is satisfied for every $0<\eps\leq\eps_0$ with $\eps_0<\frac{1}{\Lambda\sqrt{2(\sigma+2)}}$. Then
\begin{equation*}
	\delta\Psi(x,\eps y)
	\geq
	2\eps^2 A\sigma|x|^{-2\sigma-2}\left[-\Lambda^2+(\sigma+1)y_1^2\right]
\end{equation*}
for every $|x|>r>\Lambda\eps$ and $|y|<\Lambda$. Hence
\begin{equation*}
	\inf_{z\in B_\Lambda}\delta\Psi(x,\eps z)
	\geq
	2\eps^2 A\sigma|x|^{-2\sigma-2}\left[-\Lambda^2\right]
\end{equation*}
and
\begin{equation*}
	\vint_{B_1}\delta\Psi(x,\eps y)\,dy
	\geq
	2\eps^2 A\sigma|x|^{-2\sigma-2}\left[-\Lambda^2+\frac{\sigma+1}{N+2}\right],
\end{equation*}
so
\begin{equation*}
	\L_\eps^-\Psi(x)
	\geq
	A\sigma|x|^{-2\sigma-2}\left[-\Lambda^2+\beta\frac{\sigma+1}{N+2}\right]
	=\,:
	-\psi(x)
\end{equation*}
for every $|x|>r>\Lambda\eps$. Choosing large enough $\sigma$ depending on $N$, $\beta$ and $\Lambda$ we get that $\psi\leq 0$ for every $|x|>\Lambda\eps$.

Summarizing, since $\Omega=B_4(z)\setminus\overline{B_r(z)}$ with $r>\kappa\eps\geq 2\Lambda\eps$, we obtain
\begin{equation*}
	\begin{cases}
	\L_\eps^-\Psi\geq-\psi
	& \text{ in } \Omega,
	\\
	\Psi\leq u & \text{ in } \widetilde\Omega\setminus\Omega.
	\end{cases}
\end{equation*}
In what follows we recall the $\eps$-ABP estimate to show that the inequality $\Psi\leq u$ is satisfied also in $\Omega$. But before, as in the proof of Lemma~\ref{first}, we define $v=\Psi-u$ and since by assumption $\L_\eps^-u\leq 0$ in $\Omega=B_4(z)\setminus \overline{B_r(z)}\subset B_7$, we have
\begin{equation*}
	\L_\eps^+v
	\geq
	\L_\eps^-\Psi-\L_\eps^-u
	\geq
	-\psi
\end{equation*}
in $\Omega$. Thus
\begin{equation*}
	\begin{cases}
	\L_\eps^+v+\psi\geq 0 & \text{ in } \Omega,
	\\
	v\leq 0 & \text{ in } \widetilde\Omega\setminus\Omega.
	\end{cases}
\end{equation*}
By the $\eps$-ABP estimate (see Theorem~4.1 together with Remark~7.4 both from \cite{arroyobp}), 
\begin{equation*}
	\sup_\Omega v
	\leq
	\sup_{\widetilde\Omega\setminus\Omega}v
	+
	C\bigg(\sum_{Q\in\mathcal{Q}_\eps(K_v)}\Big(\sup_Q\psi^+\Big)^N|Q|\bigg)^{1/N},
\end{equation*}
where $K_v\subset\Omega$ stands for the contact set of $v$ in $\Omega$ and $\mathcal{Q}_\eps(K_v)$ is a family of disjoint cubes $Q$ of diameter $\eps/4$ such that $\overline Q\cap K_v\neq\emptyset$, so that $Q\subset\widetilde\Omega$. Since $v\leq 0$ in $\widetilde\Omega\setminus\Omega$ and $\psi\leq0$, we obtain that $v\leq 0$ in $\Omega$, that is, $\Psi\leq u$ in $\Omega$. In consequence,
\begin{equation*}
\begin{split}
	\inf_{B_1}u
	\geq
	\inf_{B_1}\Psi
	=
	~&
	\frac{3^{-2\sigma}-4^{-2\sigma}}{(r-\Lambda\eps)^{-2\sigma}-4^{-2\sigma}}\inf_{B_r(z)}u
	\\
	\geq
	~&
	(3^{-2\sigma}-4^{-2\sigma})(r-\Lambda\eps)^{2\sigma}\inf_{B_r(z)}u
	\\
	\geq
	~&
	(3^{-2\sigma}-4^{-2\sigma})\left(\frac{r}{2}\right)^{2\sigma}\inf_{B_r(z)}u
\end{split}
\end{equation*}
for every $z\in B_2$, where we have used $r>\kappa\eps\geq2\Lambda\eps$ so that $r-\Lambda\eps>\frac{r}{2}$, so the proof is finished. \qedhere
\end{proof}

Now we prove that condition (\ref{item:for-harnack}) in Lemma~\ref{lemma:apuja} holds in the desired setting.

\begin{corollary}\label{coro:cond2}
There exists $C,\sigma,\eps_0>0$ such that if $\rho\geq 0$ and $u$ is a bounded measurable function satisfying
\begin{equation*}
	\begin{cases}
	\L_\eps^-u\leq\rho & \text{ in } B_7,
	\\
	u\geq 0 & \text{ in } \R^N,
	\end{cases}
\end{equation*}
for some $0<\eps\leq\eps_0$, then
\begin{equation*}
	\inf_{B_r(z)}u
	\leq
	C\Big(r^{-2\sigma}\inf_{B_1}u+\rho\Big)
\end{equation*}
for every $z\in B_2$ and $r\in(\kappa\eps,1)$, where $\kappa=\Lambda\sqrt{2(\sigma+1)}$.
\end{corollary}

\begin{proof}
We consider $\tilde u(x)=u(x)-A\rho|x|^2$, where $A>0$ is a constant to be fixed later. Then
\begin{equation*}
	\delta\tilde u(x,\eps y)
	=
	\delta u(x,\eps y)-2\eps^2A\rho|y|^2
	\leq
	\delta u(x,\eps y),
\end{equation*}
so 
\begin{equation*}
	\inf_{z\in B_\Lambda}\delta\tilde u(x,\eps z)
	\leq
	\inf_{z\in B_\Lambda}\delta u(x,\eps z)
\end{equation*}
and
\begin{equation*}
\begin{split}
	\vint_{B_1}\delta\tilde u(x,\eps y)\,dy
	=
	~&
	\vint_{B_1}\delta u(x,\eps y)\,dy
	-
	2\eps^2A\rho\,\frac{N}{N+2},
\end{split}
\end{equation*}
where we have used that $\vint_{B_1}|y|^2\,dy=\frac{N}{N+2}$. Therefore,
\begin{equation*}
	\L_\eps^-\tilde u
	\leq
	\L_\eps^-u-A\rho\beta \,\frac{N}{N+2}
	\leq
	\left(1-A\beta \,\frac{N}{N+2}\right)\rho
	\leq
	0,
\end{equation*}
where the last inequality holds for a sufficiently large choice of $A$.

Therefore we can apply Theorem~\ref{thm.cond2} to $\tilde u$. Observe first that since $r\in(\kappa\eps,1)$ and $z\in B_2$ then $B_r(z)\subset B_3$. Thus $\tilde u\geq u-9A\rho$ in $B_r(z)$ and
\begin{equation*}
	\inf_{B_r(z)}u-9A\rho
	\leq
	\inf_{B_r(z)}\tilde u
	\leq
	Cr^{-2\sigma}\inf_{B_1}\tilde u
	\leq
	Cr^{-2\sigma}\inf_{B_1}u
\end{equation*}
and the result follows.
\end{proof}

Now we are ready to state the main result of the section.
\begin{theorem}
\label{Harnack}
There exists $C,\lambda,\eps_0>0$ such that if $u\geq 0$ in $\R^N$ is a bounded and measurable function satisfying $\L^+_\eps u\geq-\rho$ and $\L^-_\eps u\leq\rho$ in $B_7$ for some $0<\eps<\eps_0$, then
\begin{equation*}
	\sup_{B_1}u
	\leq
	C\left(\inf_{B_1}u+\rho+\eps^{2\lambda}\sup_{B_3}u\right).
\end{equation*}
\end{theorem}

\begin{proof}
By Corollary~\ref{coro:cond2} we have that $u$ satisfies condition (\ref{item:for-harnack}) in Lemma~\ref{lemma:apuja} for $\lambda=2\sigma$. 
We deduce condition (\ref{item:holder}) by taking infimum over $x,z\in B_r$ in the inequality given by Theorem~\ref{Holder}. 
We use $\varepsilon<r$ to bound $\varepsilon^\gamma<r^\gamma$.
In this way, we obtained the inequality for every $r<R/2$ and $\eps<\eps_0 R$.
We need it to hold for every $r\leq \delta R$ and $\eps<\frac{\delta}{\kappa}R$.
Therefore we have proved the result if $\delta<1/2$ and $\frac{\delta}{\kappa}<\eps_0$.
That is we have obtained the result as long as $\delta$ is small enough.
Recall that $\delta=(2^{1+2\lambda}C)^{-1/\gamma}$.
Then, it is enough to take $\gamma>0$ small enough.
We can do this since $\eps_0$, $C$, $\kappa$ and $\lambda$ only depend on $\Lambda$, $\alpha$, $\beta$ and the dimension $N$, and not on $\gamma$.
Also if Theorem~\ref{Holder} holds for a certain $\gamma>0$ it also holds with the same constants for every smaller $\gamma>0$.
\end{proof}

\begin{remark}
\label{harnack:limit}
Let $\{u_\eps\,:\,0<\eps<\eps_0\}$ be a family of nonnegative measurable solutions to the DDP with $f=0$. In view of Theorem~\ref{Holder} together with the asymptotic Arzel\'a-Ascoli theorem \cite[Lemma 4.2]{manfredipr12}, we can assume that $u_\eps\to u$ uniformly in $B_2$ as $\eps\to 0$. Then by taking the limit in the asymptotic Harnack inequality
\begin{equation*}
	\sup_{B_1}u_\eps
	\leq
	C\left(\inf_{B_1}u_\eps+\eps^{2\lambda}\sup_{B_3}u_\eps\right),
\end{equation*}
we obtain the classical inequality for the limit, that is
\begin{equation*}
	\sup_{B_1}u
	\leq
	C\inf_{B_1}u.
\end{equation*}
Similarly if $\{u_\eps\,:\,0<\eps<\eps_0\}$ is a uniformly convergent family of nonnegative measurable  functions such that $\L_\eps^+ u_{\eps}\ge -\rho$ and $\L_\eps^- u_{\eps} \le  \rho$, then for the limit we get 
\begin{align*}
\sup_{B_1}u
	\leq
	C(\inf_{B_1}u+\rho).
\end{align*}
\end{remark}

\def\cprime{$'$} \def\cprime{$'$} \def\cprime{$'$}

\end{document}